\documentclass[a4paper, 12pt]{amsart}
\usepackage{amssymb}
\usepackage{amsthm}
\usepackage{amsmath}
\usepackage{cite}
\usepackage{bbm}
\usepackage{stmaryrd} 
\usepackage{url} 
\usepackage{lineno}
\usepackage[margin=2.5cm]{geometry}
\usepackage{tikz}
\usetikzlibrary{arrows}
\usepackage{mathabx}
\usepackage{comment}

\usepackage[all,cmtip,color]{xy} 
\usepackage{float}

\usepackage[textsize=footnotesize]{todonotes}

\usepackage{hyperref}
\definecolor{dblue}{rgb}{0,0,0.70}
\hypersetup{
	unicode=true,
	colorlinks=true,
	citecolor=dblue,
	linkcolor=dblue,
	anchorcolor=dblue, 
	urlcolor   = dblue
}

\newtheorem{lemma}{Lemma}[section]
\newtheorem*{lemma*}{Lemma}
\newtheorem{theorem}[lemma]{Theorem}
\newtheorem{corollary}[lemma]{Corollary}
\newtheorem{proposition}[lemma]{Proposition}
\newtheorem{claim}[lemma]{Claim}

\newtheorem*{question*}{Question}

\newtheorem{definition}[lemma]{Definition}
\newtheorem*{definition*}{Definition}

\newtheorem{question}{Question}

\DeclareMathOperator{\On}{On}

\DeclareMathOperator{\supp}{supp}

\DeclareMathOperator{\Lim}{Lim}
\DeclareMathOperator{\Card}{Card}
\DeclareMathOperator{\Reg}{Reg}

\DeclareMathOperator{\id}{id}

\DeclareMathOperator{\Omache}{O^{Machete_\alpha}}

\DeclareMathOperator{\Ome}{O^{M\alpha}}

\DeclareMathOperator{\Mmuind}{M_{\mu-Ind}}
\DeclareMathOperator{\Mmuindtwo}{M^2_{\mu-Ind}}

\DeclareMathOperator{\Ult}{Ult}
\DeclareMathOperator{\ZFC}{\axiomft{ZFC}}

\DeclareMathOperator{\MK}{MK}

\DeclareMathOperator{\Ominf}{O^{M\infty}}
\DeclareMathOperator{\Hype}{Hype}

\newcommand{\PP}{\mathbb{P}}
\newcommand{\QQ}{\mathbb{Q}}

\newcommand{\axiomft}[1]{\mathsf{#1}}

\newcommand{\forces}{\mathrel{\Vdash}}

\DeclareMathOperator{\crit}{crit}
\newcommand{\cof}{\mathrm{Cof}}

\DeclareMathAlphabet{\mathbbold}{U}{bbold}{m}{n}

\usepackage{enumitem}

\newenvironment{enumerate-(a)}{\begin{enumerate}[label={\upshape (\alph*)}, leftmargin=2pc]}{\end{enumerate}}
\newenvironment{enumerate-(1)}{\begin{enumerate}[label={\upshape (\arabic*)}, leftmargin=2pc]}{\end{enumerate}}
\newenvironment{enumerate-(i)}{\begin{enumerate}[label={\upshape (\roman*)}, leftmargin=2pc]}{\end{enumerate}}

\def\ed{\end{document}}
\title{Of Mice and Machetes}
\author{Christopher Henney-Turner and Philip Welch}
\date{\today}
\thanks{The first-listed author wishes to acknowledge support from EPSRC Grant EP/Student 2123452, and from the Polish Academy of Science through NCN Grant WEAVE-UNISONO UMO-2021/03/Y/ST1/00281. The second-listed author wishes to acknowledge partial support for this paper from the Zukunftskolleg of the University of Konstanz, as well as from EPSRC Grant No: EP/V009001/1.}

\begin{document}
	\begin{abstract}
		Let $R$ be the class of regular cardinals which are not hyperinaccessible. We show that $L[R]$, and similar inner models in the $\alpha$-inaccessible hierarchy, can be generated by iterating a small ``machete'' mouse up through all the ordinals, and then taking a generic extension by a hyperclass Magidor iteration of Prikry forcings. We then show that such simple mice are themselves elements of $L[\Reg]$.
	\end{abstract}
	\maketitle

	\section{Introduction}
	
	In \cite{welch_HaertigPaper}, the second listed author proved the following result:
	
	\begin{quotation}
		Suppose there exists a mouse with an unbounded sequence of measurables in it. Then $L[\Card]$ is a generic extension of an iterate of the smallest such mouse.
	\end{quotation}
	
	In this paper, we shall prove an analogous result where $\Card$ is replaced by a class $A$ of regular cardinals. To do this, we divide $A$ up by Cantor-Bendixson rank:

\begin{definition} If $B$ is a class of ordinals, we write $B^*$ to denote the class of limit points of $B$.
	
	Let $A\subseteq On$ be a class of ordinals.
	For $\alpha\in \On$, $A_{\alpha}$ is the class of ordinals in $A$ of Cantor-Bendixson (CB) rank $\alpha$.
		More formally, for $\alpha\in \On$, we recursively define $A_{<\alpha}$, $A_{\geq \alpha}$ and $A_{\alpha}$ as follows:
	$$A_{<\alpha}= \bigcup_{\beta<\alpha} A_{\beta}\, ;\,
	A_{\geq \alpha} = A\setminus A_{<\alpha}
	\, ;\,A_{\alpha} = A_{\geq\alpha} \setminus A_{\geq\alpha}^{\ast}.$$
\end{definition}

A case of particular interest is when we take $A$ to be the class of strong inaccessibles.

\begin{definition}
	$\Reg =\{\delta > \omega \mid \delta \mbox{ a successor cardinal or strongly inaccessible}\}$;
we then define $\Reg_{\alpha}$ {\em etc.} as above.
	\end{definition}

The motive for including the successor cardinals in $\Reg$ is that it allows $L[\Reg]$ to identify all the cardinals of $V$. In this paper, we will assume that all inaccessibles are strongly inaccessible. Our main theorem \ref{theorem L[R]} can also be used to obtain a similar statement about the class of weak inaccessibles if this is different.

The upper bound of inaccessibles we can identify in this way is known as the class of \textit{hyperinaccessibles}.

\begin{definition}
For the purposes of this paper, a \textit{hyperinaccessible} is an inaccessible $\kappa$ such that $\kappa\in \Reg_\kappa$. $\Hype$ is the class of all hyperinaccessibles.
\end{definition}

A word of caution is warranted here, since there is some inconsistency regarding the definition of a hyperinaccessible in the literature. Some authors define hyperinaccessibles as we do here, while others call any element of $\Reg_{\geq 2}$ a hyperinaccessible.

The main goal of this paper is to prove the following:

\begin{quotation}
	Let $R=\Reg\setminus \Hype$, or $R=\Reg_{<\alpha}$ where $0<\alpha<\aleph_\omega$. 	Suppose there exist mice containing ``enough" measurables (in a sense to be specified shortly). Then $L[R]$ is a generic extension of a class length iterate of the smallest mouse with ``enough" measurables.
\end{quotation}

If we take $\alpha=1$ then this gives us the main result from \cite{welch_HaertigPaper}. The mouse used in that case is called O-Kukri, since it fits somewhere between O-Dagger and O-Sword. The mice we need to use will contain more measurables, so they will be somewhat larger than O-Kukri; but they are still strictly smaller than O-Sword. Continuing the established pattern, we shall name them after a weapon which is somewhere between a kukri and a sword in size -- a machete.

In the following definition, and indeed the rest of this paper, we shall assume the reader is familiar with the core model theory for measure sequences of order zero, as set out in \cite{zemanInnerModelsAndLargeCardinals} and used in \cite{welch_HaertigPaper}. All the mice in this paper should be taken to be below $O^{\text{Sword}}$ unless otherwise specified. Recall that a mouse $M$ is said to be \textit{active} if it has a measure on $\kappa_{M}$ its largest cardinal.\cite{zemanInnerModelsAndLargeCardinals}

\begin{definition} [$\Omache$, or $\Ome$ for short.]
	Let $\alpha\in \On$. $\Omache$ is the $<^{\ast}$-least active sound mouse $M$   whose largest measurable $\kappa_M$ 
	has Cantor-Bendixson rank $\alpha$ in the class of $M$-measurable cardinals. 
	
\end{definition}
To spell things out explicitly, $\Ome$ is the least active sound mouse $M$ whose extender sequence $E^{M}$  below  $\kappa_{M}$ contains unboundedly many measurables of all Cantor-Bendixson ranks below $\alpha$. That is, if $C\subset \kappa_{M}$ is the class of $M$-measurable cardinals, then defining (within $M$, which we may) the sets $C_{\beta}$ of CB rank $\beta$ for $\beta < \alpha$ we have that for all $	\beta<\alpha$ that $C_{\beta}$ is cofinal below  $\kappa$. We can easily see that if $M$ is the least mouse with this property, then there are no elements of $C$ of CB rank $\alpha$; otherwise we could produce a smaller sound active mouse with the same property by cutting off $M$ at such a measurable.

The natural limit of this class of mice is a mouse which contains unboundedly many measurables of \textit{all} Cantor-Bendixson ranks (up to its top measure):

\begin{definition}
	$\Ominf$ is the $<^*$-least active sound mouse $M$ such that $\kappa_M$ has CB rank $\kappa_M$ in the class of $M$-measurable cardinals.	
\end{definition}

We prove some simple properties of these mice.

\begin{proposition}\label{proposition basic properties of machetes}
\begin{enumerate}
	\item Let $\alpha \in \On$. Suppose that an active (but not necessarily sound) mouse $M$ exists with $\kappa_M$ of CB rank $\geq \alpha$. Then $\Ome$ exists, and is $\leq^*M$. Moreover, $\Ome$ contains no measurables $\leq \alpha$, and contains no measurables of CB rank $\alpha$ below $\kappa_{\Ome}$. Finally, $\rho^1_{\Ome}\leq \alpha$, and $\lvert \Ome \rvert= \max(\aleph_0,\lvert \alpha\rvert)$.
	\item Suppose that an active (but not necessarily sound) mouse $M$ exists with $\kappa_M$ of CB rank $\kappa_M$. Then $\Ominf$ exists, and is $\leq^*M$. Moreover, $\Ominf$ contains no measurables equal to their own CB rank below $\kappa_{\Ominf}$. Finally, $\rho^1_{\Ominf}=\omega$, and $\lvert \Ome \rvert= \aleph_0$.
	\item If $\Ominf$ exists, then for all $\alpha \in \On$, $\Ome$ exists and $\Ome <^* \Ominf$
\end{enumerate}
\end{proposition}
\begin{proof}
	\begin{enumerate}
	\item Without loss of generality, assume $M$ is in the $<^*$ least equivalence class containing a mouse with these properties. We will show that $[M]$ contains a sound mouse $N$ such that $\kappa_N$ has CB rank $\alpha$; it will then follow immediately that $N=\Ome$.
	
	Notice that $M$ has no measurables below $\kappa_M$ of rank $\alpha$; otherwise we could cut down $M$ at such a measurable and contradict minimality. Hence $\kappa_M$ has CB rank exactly $\alpha$. Similarly, $M$ does not contain any measurables $\lambda\leq \alpha$; or we could iterate $\lambda$, $\alpha$ many times, to produce a mouse $=^*$ to $M$ where $\kappa_M$ had CB rank $>\alpha$.
	
	Let $N$ be the transitive collapse of the $\Sigma_1$ hull of $\alpha+1$ together with $\{p_M,\kappa_M\}$. Then $\rho_N^1\leq \alpha$, and $N$ contains no measurables $\leq \alpha$. Hence $N$ is sound. By $\Sigma_1$ elementarity $N$ is active, and its top measure $\kappa_N$ has CB rank $\alpha$ in its sequence of measurables. It follows immediately that $N=\Ome$, and by construction we get that $N$ has no measurables below $\alpha$ or of CB rank $\alpha$ below $\kappa_N$, and that $\lvert N \rvert = \max(\aleph_0,\lvert \alpha,\rvert)$.
	
	\item Similar to the $\Ome$ case.
	
	\item Simply iterate the top measure of $\Ominf$ until it is above $\alpha$; the top measure of the resulting mouse $M$ is equal to its own CB rank and hence has CB rank above $\alpha$. Hence by the first part of the proposition, $\Ome$ exists and $\Ome<^* M$.
	\end{enumerate}
\end{proof}

The first main result of this paper is the following theorem:\\

\noindent{\bf Theorem \ref{theorem L[R]} (Special Case)} {\em 
	\noindent \begin{itemize}
		\item Suppose $\Ominf$ exists. Then there exists an iteration $\mathcal{I}=\langle M_i,\pi_{i,j}\rangle$ of $M_0=\Ome$, of length $\On+1$, such that $L[\Reg\setminus \Hype]$ is a generic extension of $M_{\On}$ by a hyperclass forcing.
		\item Let $\alpha\in \aleph_\omega$, and suppose $\Ome$ exists. Then there exists an iteration $\mathcal{I}=\langle M_i,\pi_{i,j}\rangle$ of $M_0=\Ome$, of length $\On+1$, such that $L[\Reg_{<\alpha}]$ is a generic extension of $M_{\On}$ by a hyperclass forcing, and another similar iteration such that $L[\Reg_{<\alpha}^w]$ is a generic extension of $M_{\On}$.
	\end{itemize}}

The machete mice defined above can be considered a special case of a more general concept: that any sufficiently simple $=^*$ equivalence classes of mice can be defined as the least mouse whose top measure satisfies some first order statement (with parameters). The threshold below which such definitions can be made is the first mouse which reflects all formulas true of its top measure down to some smaller measurable: the measurables of $M$ below $\kappa_M$ form an indescribable class.

\begin{definition}\label{Reflect}
	We define $\Mmuind$ to be the least active sound mouse whose first order $\mathcal{L}_{{\dot E}}$ formulae about its top measurable $\kappa$ all reflect down to some smaller measurable. That is, $\Mmuind$ is the least active mouse $N$ with top measurable $\kappa_{N}$ such that for any formula $\varphi(v_0,\ldots,v_n)\in \mathcal{L}_{{\dot E}}$ and for any $\gamma_1<\ldots<\gamma_n<\kappa_{N}$, if $H_{\kappa^+_N}^N\vDash \varphi(\gamma_1,\ldots,\gamma_n,\kappa_{N})$ then we can find some $\lambda\in (\gamma_n,\kappa_{N})$ such that $N$ believes $\lambda$ is measurable and $H_{\lambda^+}^N\vDash\varphi(\gamma_1,\ldots,\gamma_n,\lambda)$.
\end{definition}

\noindent Note that these formulae do not contain a predicate latter for the top measure $F_{M}$. Note also that $\Mmuind<^{\ast} O^{\text{sword}}$ (if both exist), and that $\rho^1_{\Mmuind}=\omega$. As in Proposition \ref{proposition basic properties of machetes}, if $M$ is any mouse which reflects all its formulas in the same manner as $\Mmuind$, then $\Mmuind \leq^* M$. This choice of mouse and its indescribability is merely a simple stopping point, as we could provide further arguments for indescribability with respect to higher order formulae. 

\begin{proposition}\label{proposition machete below muind}
	$\Ominf<^* \Mmuind$. Hence, $\Ome<^* \Mmuind$ for all $\alpha$.
\end{proposition}
\begin{proof}
	Clearly, $\Ominf \neq \Mmuind$: take $\varphi(v_0)$ to be ``$v_0$ has CB rank $v_0$ in the class of all measurables below $v_0$''. (Note that for any mouse $N$ and $\kappa\in N$, we can express statements about the measurables below $\kappa$ in $\mathcal{L}_{\dot{E}}$ formulas about $H^N_{\kappa^+}$ $\kappa$.) Since this property is also shared by all simple iterates of $\Ominf$ we know $\Ominf \neq^* \Mmuind$.
	
	Suppose $M<^* \Ominf$ be active. Let $C$ be the set of measurables in $M$ (including $\kappa_M$ itself). Let $\alpha<\kappa_M$ be the CB rank of $\kappa_M$. Then $\gamma:=\sup C_{>\alpha}<\kappa_M$. Let $\beta$ be the order type of $(C\setminus \gamma)_\alpha$.
	Let $\varphi(\alpha,\beta,\gamma,\kappa)$ be the formula ``The class $M$ of all measurables between $\gamma$ and $\kappa$, together with $\kappa$ itself, contains $\beta$ many elements of CB rank $\alpha$.''
	
	Then $H^M_{\kappa_M^+} \vDash \varphi(\alpha,\beta,\gamma,\kappa_M)$, but if $\lambda\in (\max(\alpha,\beta,\gamma),\kappa_M)$ is a measurable of $M$ then $H^M_{\lambda^+} \neg\vDash \varphi(\alpha,\beta,\gamma,\lambda_M)$.
\end{proof}

Our second major result is that every mouse below this threshold (and in particular, every machete mouse) can be found inside the inner model $L[\Reg]$ we have been generating fragments of. To know that the classes of regular cardinals we have just considered exist, as well as to maximise possible complexity, we adopt the hypothesis that $\On$ is Mahlo; i.e., that $\Reg$ is stationary. This means that any closed and unbounded definable class $C$ contains a regular cardinal. We adopt this hypothesis without much further remark in the rest of this paper.\\

\noindent{\bf Theorem \ref{theorem friendly machetes}} {\em Suppose that $\On$ is Mahlo.  Then for any mouse $M$, $M <^* \Mmuind $ implies $ M<^*  
K^{L[\Reg]}$.}

\section{Magidor Forcings, Iterations and the Mathias Criterion}

The forcing we shall be using in proving Theorem \ref{theorem L[R]} will be a \textit{Magidor iteration} of Prikry forcings. Before we start proving the theorem, we shall lay some groundwork about these iterations for us to use later. First, let's define what the forcing actually is. Recall that a Prikry forcing singularises a measurable cardinal, and has two relations: the usual $\leq$ and the \textit{direct extension} $\leq^*$\cite{gitik_handbook}. The Magidor iteration singularises an infinite collection of measurables, by doing a Prikry forcing on each one. If we are working in $\ZFC$, then we can do an iteration for any set of measurables.

If we are in a model of $\MK$ then we can go further and define the iteration for a proper class. In that case, the conditions of the iteration will be proper classes, and the forcing itself will be a definable hyperclass. See \cite{antosFriedman_hyperclassForcing} for details on how hyperclass forcing is formalised.

\begin{definition}\cite{magidor_IteratedPrikryForcing_IdentityCrisis}
	Let $C$ be a set of measurables (if we are working in $\ZFC$) or a class of measurables (if we are working in $\MK$), and for $\kappa\in C$ let $U_\kappa$ be a normal measure on $\kappa$. We recursively define forcings $\PP_\kappa$ for $\kappa \in C\cup \{\sup C\}$ and $\PP_\kappa$ names $\dot{\QQ}_\kappa$ of Prikry forcings for $\kappa \in C$. The two definitions depend on one another, and so are part of the same recursion. The \textit{Magidor iteration} of $C$ is the forcing $\PP_{\sup C}$.
	
	$\PP_\kappa$: For $\kappa\in C\cup \{\sup C\}$, $\PP_\kappa$ consists of sequences $p=(p_\lambda)_{\lambda\in C\cap \kappa}$, such that:
	
	\begin{itemize}
		\item For all $\lambda\in C\cap \kappa$, $p{\upharpoonleft} \lambda := (p_\beta)_{\beta\in C\cap \lambda}\in \PP_\lambda$
		\item For all $\lambda \in C\cap \kappa$, $p{\upharpoonleft} \lambda\Vdash p_\lambda \in \dot{\QQ}_\lambda$
		\item For all but finitely many $\lambda \in C\cap \kappa$, $p{\upharpoonleft} \lambda \Vdash p_\lambda\leq^*_{\dot{\QQ}_\lambda} \mathbbold{1}_{\dot{\QQ}_\lambda}$
	\end{itemize}
	
	The order $\leq$ on $\PP_\kappa$ is defined in the natural way: if $p=(p_\lambda)_{\lambda\in C\cap \kappa}$ and $q=(q_\lambda)_{\lambda\in C\cap \kappa}$ then $q \leq p$ if for all $\lambda\in C\cap \kappa$, $q{\upharpoonleft} \lambda \Vdash q_\lambda \leq_{\dot{\QQ}_\lambda} p_\lambda$. We also define a second partial order $\leq^*$ on $\PP_\kappa$: $q\leq^* p$ if for all $\lambda\in C\cap \kappa$, $q{\upharpoonleft} \lambda \Vdash q_\lambda \leq^*_{\dot{\QQ}_\lambda} p_\lambda$.
	
	$\dot{\QQ}_\kappa$: For $\kappa \in C$, let $j_{U_\kappa}$ be the ultrapower map defined by $U_\kappa$. Let $\tilde{\PP}_\kappa=j_{U_\kappa}(\PP_\kappa)$. Let $\dot{U}_\kappa^*$ be the following $\PP_\kappa$ name:
	
	$$\dot{U}_\kappa^* = \{ (\dot{A},p): {\exists q\leq^*_{\tilde{\PP}_\kappa} (j_{U_\kappa}(p) \setminus \kappa)}, \,\, p^\smallfrown q \Vdash_{\tilde{\PP}_\kappa} \check{\kappa}\in j_{U_\kappa}(\dot{A}) \}$$
	
	It can be shown\cite[2.5]{magidor_IteratedPrikryForcing_IdentityCrisis} that $\dot{U}_\kappa^*$ is a name for a normal measure on $\kappa$ in the $\PP_\kappa$ generic extension, and that $\mathbbold{1}_{\PP_\kappa} \forces \dot{U}_\kappa^*\cap V = \check{U}_\kappa$.
	
	We define $\dot{\QQ}_\kappa$ to be a name for the Prikry forcing on $\kappa$ in the generic extension defined by $\dot{U}_\kappa^*$, and define $\leq_{\dot\QQ_\kappa}$ and $\leq^*_{\dot\QQ_\kappa}$ to be (names for) the two associated partial orders.
\end{definition}

The Magidor iteration adds a cofinal $\omega$ sequence below each measurable in $C$. Like a Prikry forcing, a generic filter is determined by the generic sequence (in $\prod_{\kappa\in C} \kappa^\omega$) it adds. Notice that in $\MK$, if $C$ is a proper class then the iteration is a hyperclass forcing and a generic filter is technically a hyperclass, but the sequence it adds is just a class.

We can also say a little more about the nature of the measure $\dot{U}_\kappa^*$.

\begin{proposition}\label{proposition magidor forcing measure generating sets}\cite{benNeria_forcingMagidorOverCoreModel}
	In the forcing defined above, let $\kappa \in C$. Let $G$ be $\PP_\kappa$ generic, and for $\lambda \in C\cap \kappa$ let $G_\lambda$ be the $\omega$ sequence added by $G$ below $\lambda$. Then $(\dot{U}_\kappa^*)^G$ is generated by the measure $1$ sets of $U_\kappa$, together with the set
	
	$$\Sigma_\kappa^G := \{\nu<\kappa: \forall \lambda \in C\cap \kappa\setminus \nu, \, (\nu+1)\cap G_\lambda =\emptyset\}$$
\end{proposition}

As usual, we can prove the Prikry property:

\begin{lemma}\cite[2.1]{magidor_IteratedPrikryForcing_IdentityCrisis}
	Let $\varphi$ be some sentence (perhaps with parameters) in the language of $\PP_\kappa$, and let $p\in \PP_\kappa$. There is some $q\leq^* p$ such that either $q\forces \varphi$ or $q\forces \neg \varphi$.
\end{lemma}

The Prikry forcing on $\kappa$ consists of two components: a finite stem which is an element of $\kappa^{<\omega}$, and a measure $1$ set $X$. So a condition $p=(p_\lambda)$ of the Magidor forcing can be rearranged into a sequence of names $(\dot{s}_\lambda)$ for stems $s_\lambda\in \lambda^{<\omega}$, and a sequence of names $(\dot{X}_\lambda)$ for measure $1$ subsets $X_\lambda\subset\lambda$.

\begin{lemma}
	Let $p,q\in \PP_\kappa$. Let $p=(\dot{s}_\lambda,\dot{X}_\lambda)_{\lambda\in C\cap \kappa}$, and let $q=(\dot{t}_\lambda,\dot{Y}_\lambda)_{\lambda\in C\cap \kappa}$.
	
	Suppose that for all $\lambda$, ${p{\upharpoonleft} \lambda} \forces \dot{s}_\lambda=\dot{t}_\lambda$, or ${q{\upharpoonleft} \lambda} \forces \dot{s}_\lambda=\dot{t}_\lambda$.
	
	Then $r:=(\dot{s}_\lambda, \dot{X}_\lambda\cap \dot{Y}_\lambda)_{\lambda\in C_\cap \kappa}$ is a condition of $\PP_\kappa$, and $r\leq^* p$ and $r\leq^* q$.
\end{lemma}
\begin{proof}
	Induction on the order type of $C\cap \kappa$. The $0$ case is trivial: it boils down to saying that the intersection of two measure $1$ sets is measure $1$.
	
	Suppose $C\cap \kappa$ has a largest element $\lambda$. By inductive assumption ${r{\upharpoonleft} \lambda}\in\PP_\lambda$, and ${r{\upharpoonleft} \lambda}\leq^* {p{\upharpoonleft} \lambda}$ and ${r{\upharpoonleft} \lambda}\leq^*{ q{\upharpoonleft} \lambda}$. Then we know that ${r{\upharpoonleft} \lambda}\forces \dot{s}_\lambda=\dot{t}_\lambda\in \check{\lambda}^{<\omega}$, and that ${r{\upharpoonleft} \lambda}\forces \dot{X}_\lambda,\dot{Y}_\lambda \in \dot{U}^*_\lambda$. So ${r{\upharpoonleft} \lambda}\forces \dot{X}_\lambda \cap \dot{Y}_\lambda\in \dot{U}^*_\lambda$. Hence ${r{\upharpoonleft} \lambda}\forces ( \dot{s}_\lambda, \dot{X}_\lambda \cap \dot{Y}_\lambda) \in \dot{\QQ}_\lambda$.
	
	Hence, $r=r{\upharpoonleft} \lambda^\smallfrown (\dot{s},\dot{X}_\lambda \cap \dot{Y}_\lambda)\in \PP_\kappa$. It is now trivial to see that $r\leq^* p,q$.
	
	If $C\cap \kappa$ has a limit order type, then it is immediate that $r\in \PP_\kappa$ and $r\leq^* p,q$, just by the definition of $\PP_\kappa$ and $\leq^*$ and the inductive hypothesis.
\end{proof}

\begin{corollary}\label{Corollary Magidor chain condition}\cite[4.4]{magidor_IteratedPrikryForcing_IdentityCrisis}
	If $\kappa<\On$ then $\PP_\kappa$ satisfies the $\kappa^+$ chain condition. If we are working in $\MK$ and $\kappa=\On$ then every definable antichain is class sized: if $A \subset \PP_\kappa$ is a definable antichain, then there is a class $S$ such that $\dot{s}\in S$ if and only if $\dot{s}$ is the stem of some element of $A$. (Note that $\dot{s}$ cannot be a stem of two different elements of $A$ by the previous lemma, since $A$ is an antichain.)
\end{corollary}
\begin{proof}
	The $\kappa<\On$ case is proved by induction on the order type of $C\cap \kappa$. If the order type is a successor (say it has a largest element $\lambda$), then $\PP_\kappa=\PP_\lambda * \dot{\QQ}_\lambda$, and the result follows by inductive hypothesis and the fact that the Prikry forcing on $\lambda$ has the $\lambda^+$ chain condition. So we just need to consider the case where the order type is a limit.
	
	Suppose that this is the case, and that $A$ is an antichain of $\PP_\kappa$ of size $\kappa^+$. For each $p\in \PP_\kappa$, let $\supp(p)$ be the support of the stem of $p$, i.e. the finite set of $\lambda \in C\cap \kappa$ such that $p{\upharpoonleft} \lambda \neg\Vdash p_\lambda\leq^*_{\dot{\QQ}_\lambda} \mathbbold{1}_{\dot{\QQ}_\lambda}$. Note that this is defined in the ground model $V$. By the pigeonhole principle, we can find some finite subset $S\subset C\cap \kappa$, and some $A'\subset A$ of cardinality $\kappa^+$, such that for all $p\in A'$, $\supp(p)=S$. Since $C\cap \kappa$ has limit order type, we know that $S$ is bounded below $\kappa$, by some $\lambda$ say.
	
	Let $A''=\{{p {\upharpoonleft} \lambda}: p\in A'\}\subset \PP_\lambda$. We know that $\PP_\lambda$ has the $\lambda^+$ chain condition by assumption, so either $A''$ has cardinality less than $\lambda^+<\kappa^+$, or it contains two compatible elements. Either way, we can find $p,q\in A'$ such that ${p{\upharpoonleft} \lambda}$ and ${q{\upharpoonleft} \lambda}$ are compatible. Let $r \in \PP_\lambda$ be below both of them. Let $p',q'\in \PP_\kappa$ be obtained by sticking the parts of $p$ and $q$ (respectively) above $\lambda$ onto the end of $r$. Since $r\leq {p{\upharpoonleft} \lambda}$ and $r\leq {q{\upharpoonleft} \lambda}$, we know $p'$ and $q'$ are conditions, and that $p'\leq p$ and $q'\leq q$.
	
	But the parts of $p$ and $q$ which are above $\lambda$ have empty support. So the stems of $p'$ and $q'$ are both exactly the same as the stem of $p'{\upharpoonleft} \lambda = r = q' {\upharpoonleft} \lambda$. By the previous lemma, then, $p'$ and $q'$ are compatible. But then $p\geq p'$ and $q\geq q'$ are also compatible. Contradiction.
	
	The case $\kappa=\On$ is similar.
\end{proof}

In \cite{benNeria_MathiasForMagidor}, Ben Neria gives a generalisation of the Mathias criterion to test whether a given sequence is generic. Note that his result only applies to set-long iterations, while we are really interested in class-long sequences. We will see how to get around this issue when we come to prove Theorem \ref{theorem L[R]}; in the meantime we will restrict ourselves to set-long iterations.

\begin{theorem}\label{theorem Mathias Magidor criterion}\cite{benNeria_MathiasForMagidor}
	[$\ZFC$] Let $\PP$ be the Magidor iteration on a \textit{set} of measurables $C$, with corresponding normal measures $U_\kappa$ as above. Let $S=(S_\kappa)_{\kappa\in C} \in \prod_{\kappa \in C} \kappa^\omega$ be a sequence, not necessarily in the ground model $V$. Then the filter corresponding to $S$ is $\PP$ generic if and only if it satisfies the following two conditions:
	
	\begin{enumerate}
		\item The Mathias Criterion: For every $X\in \prod_{\kappa \in C} U_\kappa$ in $V$, the set $\bigsqcup_{\kappa \in C} S_\kappa \setminus X_\kappa$ is finite;
		\item The Separation Property: There are only finitely many tuples $\nu\leq \nu'<\kappa<\kappa'$ such that $\kappa,\kappa'\in C$ and $\nu \in S_\kappa$ and $\nu'\in S_{\kappa'}$.
	\end{enumerate}
\end{theorem}

We will use this to show that the sequence we want to add to create $L[\Reg_{<\alpha}]$ is actually generic.

More precisely, we shall use this to prove that a new criterion is sufficient for a sequence to be generic; and will then verify that the sequence we're interested in satisfies this criterion.

\begin{lemma}\label{lemma critical points Magidor generic}
	Let $M=M_0$ be a mouse (below O-Sword, to simplify notation) whose extender sequence contains a bounded set of measurables $C$. Let $\mathcal{I}=\langle M_i, \pi_{i,j}\rangle_{i\leq j\leq \theta}$ be a (set or class) length iteration of $M$. For $\lambda \in \pi_{0,\theta}(C)$, let $S_\lambda\in \lambda^\omega$ be an increasing and cofinal sequence (necessarily outside $M_\gamma$) of critical points of $\mathcal{I}$, such that if $\kappa_i\in S_\lambda$ then $\pi_{i,\theta}(\kappa_i)=\lambda$.
	
	Then $S=(S_\lambda)_{\lambda \in \pi_{0,\theta}(C)}$ is a generic sequence for the Magidor iteration of $\pi_{0\,\theta}(C)$.
\end{lemma}

\begin{proof}
	We shall show that $S$ satisfies the criteria in \ref{theorem Mathias Magidor criterion}.
	
	First, the Separation Property. Suppose that $\nu\in S_\lambda$ and $\nu'\in S_\lambda'$, and $\nu\leq \nu'<\lambda<\lambda'$. We know, by definition of $S_\lambda$, that there exists $i\in \theta$ such that $\nu=\kappa_i$ and $\pi_{i,\theta}(\kappa_i)=\lambda$. Likewise, there exists $j\in \theta$ such that $\nu'=\kappa_{j}$ and $\pi_{j,\theta}(\kappa_{j})=\lambda'$.  We know that $\kappa_i\leq \kappa_j$, so $i\leq j$. Since $\pi_{i,\theta}(\kappa_i)=\lambda< \lambda'=\pi_{j,\theta}(\kappa_j)$ we know $i\neq j$, so $i<j$. By elementarity of $\pi_{j,\theta}$,
	$$\pi_{i,j}(\kappa_i)<\pi_{j,j}(\kappa_j)=\kappa_j$$
	But $\pi_{j,\theta}{\upharpoonleft} \kappa_j = \id$, so then $\pi_{i,j}(\kappa_i)=\pi_{i,\theta}(\kappa_i)=\lambda$. This implies that $\lambda<\kappa_j$, contradicting our assumption that $\nu\leq \nu'<\lambda<\lambda'$. So there are no such interleaved pairs at all, and $S$ very much satisfies the Separation Property.
	
	We now turn to the more difficult part of this proof, showing that the Mathias Criterion is satisfied.
	
	To simplify notation, let $C^*=\pi_{0,\theta}(C)$, let $U=\{(\lambda,U_\lambda): \lambda\in C\}$ be the sequence of ultrapowers on $C$ in $M_0$, and let $U^*=\pi_{0,\theta}(U)$ be the corresponding sequence of ultrapowers on $C^*$. For $\lambda\in C^*$, let $U_\lambda^*=U^*(\lambda)$. Fix $X\in \prod_{\lambda\in C^*}U_\lambda^*$ in $M_\theta$, and for $\lambda\in C^*$ let $X_\lambda=X(\lambda)$. Let 
	
	$$\tau_X = \sup\{\lambda\in C^*: X_\lambda \neq \lambda\}\leq \sup C^*$$
	
	We want to show that $\bigsqcup S_\lambda \setminus X_\lambda$ is finite.	Suppose we can find a counterexample: an $X$ for which this union is infinite. Then let $X$ be a counterexample which minimises $\tau_X$.\footnote{Note that we haven't yet eliminated the possibility that all counterexamples $X$ satisfy $\tau_X=\sup C^*$.} Let $\bar{X}=X{\upharpoonleft} \tau_X$.
	
	Recall that by \cite[4.2.4]{zemanInnerModelsAndLargeCardinals} we can write 
	$$\bar{X}=\pi_{0,\theta}(f)(\kappa_{i_0},\ldots,\kappa_{i_n}){\upharpoonleft} \tau_X$$
	
	for some function $f\in M_0$, some $n\in \omega$ and some $i_0<\ldots<i_n<\theta$, where $\kappa_i$ denotes the $i$'th critical point of the iteration $\mathcal{I}$ as usual. (In fact, we know it's possible to do this even without including the ${\upharpoonleft}\tau_X$, but we don't need to.) Fix a way of writing this (including the ${\upharpoonleft} \tau_X$) which minimises $i_n$ for our chosen $\bar{X}$, and let $i=i_n$. (Note that we do \textit{not} require that $\pi_{0,\theta}(f)(\kappa_{i,0},\ldots,\kappa_{i_n})$ be equal to either $X$ or $\bar{X}$, only that it agree with $\bar{X}=X{\upharpoonleft} \tau_X$ up to $\tau_X$.)
	
	This is an opportune point to prove a short technical result we're going to need later.
	
	\begin{claim}
		$\kappa_i<\tau_X$
	\end{claim}
	
	\begin{proof}
		This is trivial if $\tau_X>\kappa_j$ for all $j<\theta$. Suppose otherwise, and let $j<\theta$ be least such that $\kappa_j\geq \tau_X$. Now $\bar{X}\in M_\theta$ can be coded easily as a subset of $\tau_X$ and hence as a subset of $\kappa_j$. By \cite[4.2.1]{zemanInnerModelsAndLargeCardinals}, we know that this coding already exists in $M_j$, and hence that $\bar{X}\in M_j$. Since $\pi_{j,\theta}$ acts as the identity on ordinals below $\kappa$, it is also easy to see that $\pi_{j,\theta}(\bar{X}){\upharpoonleft} \tau_X=\bar{X}$.
		
		We can write $\bar{X}=\pi_{0,j}(g)(\kappa_{j_0},\ldots,\kappa_{j_m})$ for some $g\in M_0$, some $m\in \omega$ and some $j_0<\ldots<j_m<j$. But then
		
		$$\bar{X}=\pi_{0,\theta}(g)(\kappa_{j_0},\ldots,\kappa_{j_m}){\upharpoonleft} \tau_X$$
		
		Hence $i\leq j_m<\tau_X$.
	\end{proof}
	
	Returning from this diversion, we now prove the central claim of this lemma.
	
	\begin{claim}\label{claim critical points in measure 1 set}
		Let $i<j<\theta$, and suppose that $\pi_{j,\theta}(\kappa_j)=\lambda\in C^*$. Then $\kappa_j\in X_\lambda$.
	\end{claim}
	
	\begin{proof}
		If $\lambda \geq \tau_X$, then $X_\lambda=\lambda$ and so the claim is trivial. So suppose $\lambda \in C^* \cap \tau_X$.
		
		Since $j+1>i$, we know $Y:=\pi_{0,j+1}(f)(\kappa_{i_0},\ldots,\kappa_{i_n})$ is well defined. Moreover, $\pi_{j+1,\theta}(Y)\upharpoonleft\tau_X=\bar{X}$, so if $\kappa \in \pi_{0,j+1}(C)$ and $\pi_{j+1,\theta}(\kappa)<\tau_X$ then $Y(\kappa)$ is well-defined, and is $M_{j+1}$ believes it is a measure $1$ subset of $\kappa$; and $\pi_{j+1,\theta}(Y(\kappa))=X_{\pi{j+1,\theta}(\kappa)}$. Now $\pi_{j,j+1}(\kappa_j)$ is such a $\kappa$, since $\pi_{j+1,\theta}(\pi_{j,j+1}(\kappa_j))=\lambda<\tau_X$
		
		$M_{j+1}$ believes that $Y(\pi_{j,j+1}(\kappa_j)) \subset \pi_{j,j+1}(\kappa_j)$ is measure 1. But (since $\kappa_j$ is the critical point the iteration $\mathcal{I}$ at stage $j$) we know that we generated $M_{j+1}$ by taking the ultrapower of $\kappa_j$ in $M_j$, and $\pi_{j,j+1}$ is the corresponding ultrapower map. So $Y(\pi_{j,j+1}(\kappa_j))$ being measure $1$ in $M_{j+1}$ implies $\kappa_j\in Y(\pi_{j,j+1}(\kappa_j))$.
		
		The critical point of $\pi_{j+1,\theta}$ is $\kappa_{j+1}>\kappa_j$. So
		
		$$\kappa_j=\pi_{j+1,\theta}(\kappa_j)\in \pi_{j+1,\theta}(Y(\pi_{j,j+1}(\kappa_j)))=X_\lambda$$
		
		So $\kappa_j\in X_\lambda$.
	\end{proof}
		
	Let $\lambda \in C^*\cap \tau_X$. It follows immediately from the second claim that $S_\lambda \setminus X_\lambda \subset \kappa_i$. Of course, this tells us nothing if $\lambda\leq \kappa_i$, but if $\lambda > \kappa_i$ then since $S_\lambda$ is cofinal in $\lambda$, we know that $S_\lambda \setminus X_\lambda$ is finite. Moreover, if the least $j<\theta$ such that $\kappa_j\in S_\lambda$ is greater than $i$, then $S_\lambda \setminus X_\lambda = \emptyset$.
	
	We shall show that this second, stronger statement holds for all but finitely many $\lambda \in C^* \setminus \kappa_i$. That is, there are only finitely many $\lambda>\kappa_i$ in $C^*$ such that $\kappa_j\in S_\lambda\setminus X_\lambda$ for some $j<i$. Suppose there were infinitely many such elements of $C^*$. Call them $\kappa_i<\lambda_0<\lambda_1<\ldots$ say, and call the corresponding indices $j_0,j_1,\ldots$. By the second claim, we know that for all $m$, $\kappa_{j_m}<\kappa_i$, and by assumption $\lambda_m>\kappa_i$. So if $m<n$ and $\kappa_{j_m}<\kappa_{j_n}$, we have $\kappa_{j_m}<\kappa_{j_n}<\lambda_m<\lambda_n$ -- and while proving the separation property, we showed that this state of affairs is impossible. Hence $\kappa_{j_m}>\kappa_{j_n}$. But now we have an infinite descending chain of ordinals. So our assumption about the existence of the $\lambda_n$'s and $j_n$'s is contradictory -- there are only finitely many $\lambda>\kappa_i$ such that we can find $j<i$ with $\kappa_j\in S_\lambda$. And as we saw before, for any single such $\lambda$, $S_\lambda \setminus X_\lambda$ is finite.
	
	Hence, the set $\bigsqcup_{\lambda \in C^* \setminus \kappa_i} S_\lambda \setminus X_\lambda$ is finite (note that $\kappa_i$ itself is not in $C^*$). We are assuming that $\bigsqcup_{C^*} S_\lambda \setminus X_\lambda$ is infinite, so $\bigsqcup_{\lambda \in C^* \cap \kappa_i} S_\lambda \setminus X_\lambda$ must be infinite.
	
	Let us define a new sequence $\tilde{X} \in \prod_{C^*}U^*_\lambda$:
	
	\begin{align*}
		\tilde{X}(\lambda) = 
		\begin{cases}
			X(\lambda) & \text{if } \lambda \in C^*\cap \kappa_i\\
			\lambda & \text{if } \lambda \in C^* \setminus \kappa_i
		\end{cases}
	\end{align*}
	
	Let $\tilde{X}_\lambda=\tilde{X}(\lambda)$. We have just seen that $\bigsqcup_{\lambda \in C^*} S_\lambda \setminus \tilde{X}_\lambda = \bigsqcup_{\lambda \in C^* \cap \kappa_i} S_\lambda \setminus X_\lambda$ is infinite. Also, $\tau_{\tilde{X}} := \sup \{\lambda: \tilde{X}_\lambda \neq \lambda\} \leq \kappa_i$.
	
	We saw earlier that $\kappa_i < \tau_X$, so $\tau_{\tilde{X}}<\tau_X$. But then the existence of $\tilde{X}$ contradicts minimality of $\tau_X$. So we have a contradiction, and there is no $X$ such that $\bigsqcup_{C^*} S_\lambda\setminus X_\lambda$ is infinite.
	
	So the Mathias Criterion and Separation Property both hold, and hence $S$ is generic by Theorem \ref{theorem Mathias Magidor criterion}.
\end{proof}

\section{$V$ cardinals}

Before starting the argument proper, we will briefly review a couple of useful standard lemmas about fixed points of iterations.

\begin{lemma}\label{Lemma V cardinals are limits of measure sequences}\cite[6.1.5]{zemanInnerModelsAndLargeCardinals}
	Let $\lambda$ be a (weak) inaccessible of $V$, or $\lambda = \On$. Let $M_0$ be a mouse, such that $\lambda>\lvert M_0\rvert$. Let $\mathcal{I}=\langle M_i,\pi_{i,j}\rangle$ be a simple iteration of length $\lambda$ of $M_0$ with no nontrivial critical points, whose set/class of critical points is cofinal below $\lambda$, but does not go above $\lambda$. Then there exists a cofinal increasing sequence $(i_\epsilon)_{\epsilon<\lambda}$ such that for all $\epsilon<\beta<\lambda$,
	
	$$\pi_{i_\epsilon,i_\beta}(\kappa_{i_\epsilon})=\kappa_{i_\beta}$$
\end{lemma}

Using a standard argument, we can use this to prove that if we iterate a measure cofinally below any cardinal $\lambda$ of $V$, then in the resulting model $\lambda$ will be measurable.

\begin{lemma}\label{Lemma V cardinals end up with a measurable on them}
	Let $\lambda$, $M_0$ and $\mathcal{I}$ be as in the previous lemma. If $\lambda\neq \On$, let $M_\lambda$ be the new mouse in the continuation of $\mathcal{I}$ (i.e.  the transitive collapse of the direct limit of $\langle M_i,\pi_{i,j}\rangle_{i\leq j<\lambda}$). Let $\pi_{i,\lambda}:M_i\rightarrow M_\lambda$ be the corresponding maps. Then $\lambda$ is measurable in $M_\lambda$; specifically,
	
	$$\pi_{i_0,\lambda}(\kappa_{i_0})=\lambda$$
	
	Similarly, if $\lambda=\On$ (and so we don't have a well defined transitive collapse) then let $\bar{M}_\lambda$ be the structure which is the direct limit of the system $\langle M_i,\pi_{i,j}\rangle$ and let $\bar{\pi}_{i,\lambda} :M_i\rightarrow \bar{M}_\lambda$ be the corresponding maps. Then $\bar{\pi}_{i_0,\lambda}(\kappa_{i_0})$ is an ordinal in the sense of $\bar{M}_\lambda$, and has order type $\On$.
\end{lemma}
\begin{proof}
	First we prove the case $\lambda \neq \On$. Clearly $\pi_{i_0,\lambda}(\kappa_{i_0})\geq \pi_{i_0,i_\epsilon}(\kappa_{i_0})=\kappa_{i_\epsilon}$ for all $\epsilon<\lambda$, and hence $\pi_{i_0,\lambda}(\kappa_{i_0})\geq \lambda$. On the other hand, suppose $\pi_{i_0,\lambda}(\kappa_{i_0})>\lambda$. For large enough $i_\epsilon<\lambda$, we can find a preimage $\bar{\lambda}=\pi_{i_\epsilon,\lambda}^{-1}(\lambda)\in M_{i_\epsilon}$ of $\lambda$ in $M_{i_\epsilon}$. By elementarity $\bar{\lambda}<\pi_{i_0,i_\epsilon}(\kappa_{i_0})=\kappa_{i_\epsilon}<\lambda$. But then $\pi_{i_\epsilon,\lambda}(\bar{\lambda})=\bar{\lambda}<\lambda$. Contradiction.
	
	The case $\lambda=\On$ is proved in exactly the same way; the notation is just a bit messier.
\end{proof}

Finally, we can also prove that if we did \textit{not} iterate any measurable below $\lambda$ cofinally often, then $\lambda$ will be a fixed point of the iteration.

\begin{lemma}\label{lemma V cardinals fixed points}
	Let $\lambda$ be an inaccessible of $V$, let $\mathcal{I}=\langle M_i\rangle_{i<\theta}$ be a set-length simple iteration of mice or weasels which all contain $\lambda$, whose critical points are all below $\lambda$. Suppose that there is no sequence of the sort described in Lemma \ref{Lemma V cardinals are limits of measure sequences}, i.e. no sequence $(i_\epsilon)_{\epsilon<\theta}$ of ordinals below $\theta$ such that for all $\epsilon<\beta<\theta$, $\pi_{i_\epsilon,i_\beta}(\kappa_{i_\epsilon})=\kappa_{i_\beta}$, and such that $(\kappa_{i_\epsilon})_{\epsilon<\theta}$ is cofinal below $\lambda$. Then $\lambda$ is a fixed point of $\mathcal{I}$.
\end{lemma}

The lemma follows from simple cardinal arithmetic; the details are left to the reader.

\section{Generating an inner model}

With those preliminaries done, we now return to O-Machete and Theorem \ref{theorem L[R]}. We shall actually prove a slightly more general result, of which the $\Reg$ statement given in the introduction is a special case:

\begin{theorem}\label{theorem L[R]}
	Let $R$ be a class of regular uncountable cardinals of $V$, and let $\alpha \leq \infty$. Suppose that $\Ome$ exists, and either:
	\begin{itemize}
		\item $\alpha<\infty$, none of the elements of $R$ have Cantor-Bendixson rank $\alpha$ in $R$, and $R$ has at most finitely many elements below $\alpha$; or
		\item $\alpha=\infty$ and $R$ has at most finitely many elements which are equal to their own Cantor-Bendixson rank in $R$.
	\end{itemize}
	Then there is a class long iterate $M_{\On}$ of $\Ome$ such that $L[R]$ is a hyperclass generic extension of $M_{\On}$.
\end{theorem}

We get the first version of the theorem given in the introduction by taking $\alpha=\infty$ and $R=\Reg\setminus \Hype$, and the second by taking $R=\Reg_{<\alpha}$. In the former case, if $\kappa \in \Reg$ has Cantor-Bendixson rank $\kappa$ in $\Reg$, then by definition $\kappa \in \Hype$. In the latter case, the first $\omega$ elements of $R$ are the uncountable cardinals below $\aleph_\omega$, and so the assumption that $\alpha<\aleph_\omega$ tells us there are only finitely many elements of $R$ below $\alpha$.

The rest of this section will be devoted to proving the theorem.

\subsection{Iterating O-Machete}
\noindent \\
For $\beta<\alpha$, let $R_\beta$ be the class of elements of $R$ of rank $\beta$. (So if $R=\Reg_{<\alpha}$ then $R_\beta=\Reg_\beta$.) Let $W_\beta$ be the class of $\omega$ limits of $R_\beta$. For $\lambda\in W_\beta$, let $S_\beta^\lambda$ be the $\omega$ sequence of cardinals in $R_\beta$ immediately preceding $\lambda$: i.e. $S_\beta^\lambda$ is the $\omega$ many elements of $(R_\beta\cap \lambda) \setminus \sup (W_\beta \cap \lambda)$, ordered as an increasing sequence. Notice that $\bigcup_{\lambda\in W_\beta}S_\beta^\lambda\subset R_\beta$ differs from $R_\beta$ only by a finite tail. Finally, let $W=\bigcup_{\beta<\alpha}W_\beta$.
	
	Let $M_0=\Ome$. We saw in \ref{proposition basic properties of machetes} that the extender sequence of $M_0$ contains unboundedly many measurables of every rank below $\alpha$, but none of rank $\alpha$ or above (apart from its top measure). Let $C$ denote the collection of all the measurables in the extender sequence of $M_0$ (below its top measure).
	
	Let $\mathcal{I}=\langle M_i, \pi_{i,j}\rangle$ be the $\On+1$ long iteration which is defined by the following rule:
	
	\begin{quotation}
		If $i\in \On$ and $M_i$ is defined, then the critical point used to generate $M_{i+1}$ is the first measurable $\kappa$ in $\pi_{0,i}(C)$ which is not an element of $W$. If no such measurables exist, then instead the critical point used is the top measure of $M_i$. As usual, if $i\in \Lim$ then we take direct limits.
	\end{quotation}
	
	$M_{\On}$ is then the cut-down of the direct limit of $\langle M_i\rangle_{i<\On}$ to the height of the ordinals.
	
	It is immediately clear that we can follow this rule in any iterate $M_i$, so this gives us a well-defined iteration with no trivial stages.
	
	\begin{lemma}\label{lemma critical points line up with Reg}
		\noindent
		\begin{enumerate}
			\item Let $\lambda<\beta<\alpha$, $\lambda \in R_\beta$, and if $\alpha<\infty$ then suppose $\lambda>\alpha$. Then $\lambda$ is a critical point $\kappa_i$ of the iteration $\mathcal{I}$, and $\pi_{i,\On}(\lambda)$ is the least element of $W_\beta$ above $\lambda$ (or if $W_\beta\subset \lambda$ then $\lambda$ is iterated out of $M_{\On}$ and $\pi_{i,\On}(\lambda)$ is undefined.)
			\item $\pi_{0,\On}(C)=\bigcup_{\gamma<\alpha}W_\gamma$
		\end{enumerate}
	\end{lemma}
	\begin{proof}
		Let $\lambda \in R_\beta$ be as described. Let $i<\On$ be least such that $\kappa_i\geq \lambda$.
		
		\begin{claim}
			$i=\lambda$
		\end{claim}
		\begin{proof}
			By \ref{proposition basic properties of machetes}, if $\alpha=\infty$ then $\Ome$ is countable, and otherwise $\Ome$ has cardinality equal to $\max(\omega,\lvert \alpha \rvert)$. In either case, $\lvert \Ome \rvert \leq \lambda$. On the other hand, obviously $M_i$ has $V$ cardinality at least $\lambda$. Let $j\leq i$ be least such that $\lvert M_j\rvert \geq \lambda$. It is easy to see that for $\gamma$ an infinite ordinal,  $\lvert J_\gamma^E\rvert=\lvert \gamma \rvert$, so $j$ is also least such that $\lambda \in M_j$.
			
			Since taking a single ultrapower does not increase the $V$ cardinality of a mouse, we know that $j$ is not a successor cardinal. So $M_j$ is a direct limit model, and for some $k<j$, there exists $\bar{\lambda}\in M_k$ such that $\pi_{k,j}(\bar{\lambda})=\lambda$.
			
			The sequence $\{\pi_{k,j'}(\bar{\lambda}): k<j'<j\}$ is cofinal below $\lambda$. Since $\lambda$ is regular, it follows that $\lambda \leq j \leq i$. On the other hand, by definition of $i$, there is an $i$ long sequence of critical points below $\lambda$, so $i\leq \lambda$.
		\end{proof}
		
		So the sequence of critical points $(\kappa_j)_{j<i}$ of the iteration $M_0\rightarrow M_i$ is a $\lambda$ long increasing sequence below $\lambda$, and is therefore unbounded. We've also already seen $\On\cap \Ome<\lambda$. By Lemma \ref{Lemma V cardinals end up with a measurable on them}, in $M_i=M_\lambda$ there is a measurable on $\lambda$ and $\lambda\in\pi_{0,i}(C)$.
		
		It is now easy to see that the rule for constructing the iteration makes $\lambda$ the $i$'th critical point $\kappa_i$, because if $\mu \in \pi_{0,i}(C)\cap \lambda$ then for some $j<i$, $\mu \in \pi_{0,j}(C)$ and $\kappa_j>\mu$; hence $\mu\in W$. On the other hand $\lambda \not \in W$ since all the elements of $W$ have $V$ cofinality $\omega$, while $\lambda\in R$ is regular in $V$.
		
		It remains to show that $\pi_{i,\On}(\lambda)$ is the smallest element of $W_\beta$ above $\lambda$ (or does not exist if $\lambda>\sup W_\beta$). We do this by proving the following two statements together, using induction on $\beta<\alpha$ but simultaneously for all $\lambda$:
		
		\begin{claim}
			\noindent
			\begin{enumerate}
				\item \label{induction critical points reg part 1} If $\lambda'$ is the immediate successor of $\lambda$ in $R_\beta$, and $\lambda=\kappa_i$ and $\lambda'=\kappa_j$, then $\pi_{i,j}(\lambda)=\lambda'$.\footnote{Of course, we've just shown that in fact $i=\lambda$ and $j=\lambda'$, but we use $i$ and $j$ to emphasise that we're thinking of the ordinals as stages of the iteration.}
				\item \label{induction critical points reg part 2} If $\mu=\min W_\beta\setminus \lambda$ and $\lambda=\kappa_i$, then $\pi_{i,\On}(\lambda)=\mu$. If $W_\epsilon \setminus \lambda=\emptyset$ then $\pi_{i,j}(\lambda)$ is unbounded as $j$ increases.
				\item \label{induction critical points reg part 3} $W_\beta \subset\pi_{0,\On}(C)$.
			\end{enumerate}
		\end{claim}
		\begin{proof}
			Induction on $\beta$. Suppose that both \ref{induction critical points reg part 1} and \ref{induction critical points reg part 2} hold for all $\gamma<\beta$.
			
			\ref{induction critical points reg part 1}: Note that $\lambda>\beta$: if $\alpha<\infty$ then this is immediate as $\lambda>\alpha$; and if $\alpha=\infty$ then it comes out of the assumptions about $R$.
			
			We know that $\lambda=\kappa_i$ is the critical point at stage $i$. The first step is to show that its Cantor-Bendixson rank in the class $\pi_{0,i}(C)$ of measurables of $M_i$ is exactly $\beta$. Since we are assuming \ref{induction critical points reg part 3} holds for all $\gamma<\beta$, and since $\pi_{0,\On}(C)\upharpoonleft \kappa_i = \pi_{0,i}(C)\upharpoonleft \kappa_i$, we know that $\pi_{0,i}(C)$ contains all of $\bigcup_{\gamma<\beta}W_\gamma$, and hence its C-B rank must be at least $\beta$.
						
			On the other hand, since $\lambda$ is the critical point at stage $i$ of the iteration, we know that every element of $\pi_{0,i}(C)$ below $\lambda$ is in $W_\gamma$ for some $\gamma<\alpha$. And $\lambda$ has Cantor-Bendixon rank $\beta$ in the class $\bigcup W_\gamma$. So its rank in $\pi_{0,i}(C)$ is at most $\beta$, and hence it is exactly $\beta$.
			
			Let $\nu=\sup(W_\beta\cap \lambda)<\lambda$. We have shown that $\lambda$ is the least measurable in $\pi_{0,i}(C)$ which is strictly above $\nu$ and has Cantor Bendixson rank $\beta$ in $\pi_{0,i}(C)$. The same argument shows that $\lambda'$ is the least measurable in $\pi_{0,j}(C)$ which is strictly above $\sup(W_\beta \cap \lambda')$ and has rank $\beta$.
			
			But $W_\beta$ is a collection of $\omega$ limits of $R_\beta$, and $\lambda'$ is the immediate successor of $\lambda$ in $R_\beta$, so $\sup(W_\beta \cap \lambda')=\nu$. And since $\nu,\beta<\lambda=\kappa_i$, we know that $\pi_{i,j}(\nu)=\nu$ and $\pi_{i,j}(\beta)=\beta$. Since $\pi_{i,j}$ is elementary, $\pi_{i,j}(\lambda)=\lambda'$ as required.
			
			\ref{induction critical points reg part 2}:	Suppose that $\lambda<\sup R_\beta$. Let $\lambda=\kappa_i=\kappa_{i_0}<\kappa_{i_1}<\ldots$ be the first $\omega$ many elements of $R_\beta$ above $\lambda$, whose supremum $\mu$ is the least element of $W_\beta$ above $\lambda$. Let $j=\sup \{i_n: n\in \omega\}$. By induction on \ref{induction critical points reg part 1} we know $\pi_{i_0,i_n}(\lambda)=\kappa_{i_n}$ for all $n\in \omega$. So since $M_j$ is a direct limit model, $\pi_{i,j}(\lambda)=\sup \{\kappa_n: n\in \omega\}=\mu$.
			
			By stage $j$ of the iteration $\mathcal{I}$, we have seen unboundedly large critical points below $\mu$ (namely, $\kappa_{i_0},\kappa_{i_1},\ldots$). So from $j$ onwards, the remainder of the iteration has no critical points below $\mu$. Moreover, $\mu\in W_\beta$ and we defined $\mathcal{I}$ such that no element of $W_\beta$ would ever be a critical point. So all the critical points of the iteration from $M_j$ to $M_{\On}$ are strictly greater than $\mu$. Hence $\pi_{j,\On}(\mu)=\mu$ and so $\pi_{i,\On}(\lambda)=\mu$.
			
			Now suppose $\lambda \geq \sup W_\beta$. As we saw earlier, $\lambda$ is of rank at least $\beta$ in $\pi_{0,i}(C)$. This will also be true (by elementarity) about $\pi_{i,j}(\lambda)$ in $\pi_{0,j}(C)$, for $i<j<\On$. So for all such $j$, we know $M_j$ will contain cofinally many measurables $\kappa \leq \pi_{i,j}(\lambda)$ of $\pi_{0,j}(C)$ which are not elements of $\bigcup_{\gamma<\beta} W_\gamma$. Since $\bigcup_{\gamma<\beta}W_\gamma$ agrees with $\bigcup_{\gamma<\alpha}W_\gamma$ above $\lambda$, it follows that sufficiently large such $\kappa$ are not in $\bigcup_{\gamma<\alpha}W_\gamma$, and hence $\kappa_j\leq \kappa < \pi_{0,j}(\lambda)$. But the $\kappa_j$ form a class long increasing sequence, and hence $\pi_{i,j}(\lambda)$ increases unboundedly as $j$ increases and $\pi_{i,\On}(\lambda)$ is undefined.
			
			\ref{induction critical points reg part 3} Let $\mu \in W_\beta$. Let $\lambda=\kappa_i$ be a cardinal in the $\omega$ sequence of $R_\beta$ immediately preceding $\mu$. Then $\mu=\min W_\beta \setminus \lambda$, and hence by \ref{induction critical points reg part 2} $\pi_{i,\On}(\lambda)=\mu$. But $\lambda=\kappa_i \in \pi_{0,i}(C)$ (since it is a critical point, and hence measurable in $M_i$), and hence $\pi_{i,\On}(\lambda)\in \pi_{0,\On}(C)$.
		\end{proof}
		
		This claim completes the proof of the first part of the lemma, and also shows that $\pi_{0,\On}(C)\supset \bigcup_{\gamma<\alpha}W_\gamma$. All that remains is to see that $\pi_{0,\On}(C)\subset \bigcup_{\gamma<\alpha}W_\gamma$. This follows almost immediately from the definition of $\mathcal{I}$. Suppose that $\lambda \in \pi_{0,\On}(C)$. Then for large enough $i$, there is some $\bar{\lambda}\leq \lambda$ such that $\pi_{i,\On}(\bar{\lambda})=\lambda$. If $i$ is also large enough that $\kappa_i>\lambda$, then we have $\bar{\lambda}=\lambda$, and hence $\lambda \in \pi_{0,i}(C)$. But since $\kappa_i>\lambda$, the definition of $\mathcal{I}$ tells us that $\lambda \in \bigcup_{\gamma<\alpha}W_\gamma$.
	\end{proof}
	
	\begin{corollary}\label{Corollary set size machete genericity}
		Let $\kappa \in \pi_{0,\On}(C)$. In $M_{\On}$, let $\PP_\kappa$ be the Magidor iteration of $\pi_{0,\On}(C){\upharpoonleft} \kappa$. Then
		
		$$T^\kappa:=\{(\lambda,S_\beta^\lambda): \beta<\alpha,\lambda \in W_\beta \cap \kappa\} $$
		
		is generic for $\PP_\kappa$.
	\end{corollary}
	\begin{proof}
		Since Lemma \ref{lemma critical points line up with Reg} holds for all but finitely many terms of $T^\kappa$, we know that $T^\kappa$ satisfies the condition for genericity proved in Lemma \ref{lemma critical points Magidor generic}.
	\end{proof}
	
	\subsection{The Hyperclass Generic Extension}
	\noindent \\
	So we now have our generic extension. There is a subtle issue we still need to deal with, however. The corollary only talks about $\PP_\kappa$, a set size initial segment of the Magidor iteration. We want to show that it holds for the Magidor iteration of the whole of $\pi_{0\,\On}(C)$. But this is a proper class of $M_{\On}$, and so the Magidor iteration is a hyperclass forcing (in the sense defined in \cite{antosFriedman_hyperclassForcing}). Lemma \ref{lemma critical points Magidor generic} only tells us about set size forcings. We must find a way to get around this difficulty if we want to force with the full iteration $\PP$. First, we must check that $\PP$ satisfies the conditions from \cite{antosFriedman_hyperclassForcing} which make forcing work properly.
	
	Recall from \cite{antosFriedman_hyperclassForcing} that a hyperclass forcing is a definable collection $\PP$ of classes (of a model of $\MK$) which is a partial order with a maximal element. $\PP$ is \textit{pretame} if for every class-long sequence $\langle D_i: i\in A\rangle$ of dense definable subhyperclasses of $\PP$, and for every $P\in \PP$, there is some $Q\leq P$ and some class size $D_i' \subset D_i: i\in A$ such that every $D_i'$ is predense below $Q$.
	
	The theory $\MK^{**}$, as defined in \cite{antosFriedman_hyperclassForcing}, consists of the axioms of Morse-Kelley class theory together with the Class-Bounding axiom:
	\begin{equation*}
		\forall x \exists A \varphi(x,A)\rightarrow \exists B \forall x \exists y \varphi(x,(B)_y)
	\end{equation*}
	and Dependent Choice for Classes:
	\begin{equation*}
		\forall \vec{X} \exists Y \varphi(\vec{X},Y) \rightarrow \forall X \exists \vec{Z} (Z_0 = X \wedge \forall i\in \On \varphi(\vec{Z}\upharpoonleft i, Z_i))
	\end{equation*}
	
	\cite{antosFriedman_hyperclassForcing} shows that if a hyperclass forcing $\PP$ in a model $M$ of $\MK^{**}$ is pretame, and $G$ is a $\PP$-generic filter, then we can define a generic extension $M[G]$ which is a model of $\MK^{**}$. It is easy to check that a class-length iteration of mice gives us a model of $\MK^{**}$:
	
	\begin{proposition}
		Let $j$ be the uncollapsing map from $M_{\On}$ to the direct limit $H$ of $\langle M_i\rangle_{i<\On}$. Let $\mathcal{C}=\{j^{-1}(X): X\in H, H\vDash X \in \mathcal{P}(j''M_{\On})\}$. Then
		$$\langle M_{\On},\mathcal{C}\rangle\vDash \MK^{**}$$
	\end{proposition}
	\begin{proof}
		It is well-known (see \cite[page 478]{marekMostowski_MK}) that $\langle M_{\On},\mathcal{C}\rangle\vDash \MK$, and it is easy to check that the structure also believes Class Bounding and Dependent Choice for Classes.
	\end{proof}	
	
	\begin{lemma}
		Let $\PP$ be the Magidor iteration of $\pi_{0,\On}(C)$, defined in $\langle M_{\On},\mathcal{C}\rangle$. Then $\PP$ is pretame.
	\end{lemma}
	\begin{proof}
		Follows immediately from the fact that it has the ``$\On^+$ chain condition" (see Corollary \ref{Corollary Magidor chain condition}; the proof goes through for classes in exactly the same way), so every dense subhyperclass of $\PP$ contains a predense subclass.
	\end{proof}
	
	This means that we can use the results of \cite{antosFriedman_hyperclassForcing} to define the concept of a generic extension by the hyperclass forcing $\PP$. As usual for Magidor iterations, the generic filter will be defined by a sequence. Although the filter itself is a proper hyperclass, the sequence defining it will be a class. More precisely, it will be a class function taking every element $\mu$ of $\pi_{0,\On}(C)=\bigcup_{\beta<\alpha}R_\beta$ to a cofinal $\omega$ sequence below $\mu$.
	
	\begin{lemma}
		Let $\PP$ be as above. Then
		
		$$T:=\{(\mu,S_\beta^\mu) : \beta<\alpha,\mu \in W_\beta\}$$
		
		is generic for $\PP$.
	\end{lemma}
	\begin{proof}
		Let $(i_\beta)_{\beta\in \On}$ be the sequence given in \ref{Lemma V cardinals are limits of measure sequences} for $\lambda:=\On$. For $\beta \in \On$, let 
		
		$$\mathcal{C}_\beta=(\mathcal{P}(H_{\kappa_{i_\beta}}))^{M_{i_\beta}}$$
		
		and let
		
		$$\mathcal{H}_\beta=\langle H_{\kappa_{i_\beta}}^{M_{i_\beta}}, \mathcal{C}_\beta\rangle$$
		
		We know that the analogous structure $\langle M_{\On},\mathcal{C}\rangle$ in the direct limit of $\langle M_i\rangle_{i<\On}$ is a model of $\MK^{**}$ so by elementarity $\mathcal{H}_\beta$ is a model of $\MK^{**}$, with set-part $H_{\kappa_{i_\beta}}^{M_{i_\beta}}$. So we can also view $\PP{\upharpoonleft} \kappa_{i_\beta}=\PP_{\kappa_{i_\beta}}$ as a hyperclass forcing over $\mathcal{H}_\beta$.
		
		\begin{claim}
			Let $\beta\in \On$. A filter $G\subset \PP_{\kappa_{i_\beta}}$ is generic over $\PP_{\kappa_{i_\beta}}$ for $\mathcal{H}_\beta$ (as a hyperclass forcing) if and only if it is generic over $\PP_{\kappa_{i_\beta}}$ for $M_{\On}$ (as a set forcing).
		\end{claim}
		\begin{proof}
			It is easy to verify that $G$ is a filter in the sense of $\mathcal{H}_\beta$ if and only if it is a filter in the sense of $M_{\On}$, so the claim is coherent. To keep notation tidier, let us write $\bar{\PP}$ for $\PP_{\kappa_{i_\beta}}$ just for this claim.
			
			The critical points of the iteration from $M_{i_\beta}$ to $M_{\On}$ are all $\geq \kappa_{i_\beta}$. So $M_{\On}$ and $M_{i_\beta}$ agree on $H_{\kappa_{i_\beta}}$ and its subsets. In particular, $\mathcal{H}_\beta$ is a set in $M_{\On}$:
			
			$$\mathcal{H}_\beta = \langle H_{\kappa_{i_\beta}}^{M_{\On}}, (\mathcal{P}(H_{\kappa_{i_\beta}}))^{M_{\On}}\rangle \in M_{\On}$$
			
			Hence any definable dense subhyperclass of $\bar{\PP}$ in $\mathcal{H}_\beta$ is a set in $M_{\On}$. So a filter which is generic in the sense of $M_{\On}$ is generic in the sense of $\mathcal{H}_\beta$ as well.
			
			On the other hand, we saw in Corollary \ref{Corollary Magidor chain condition} that $\bar{\PP}$ has the $\kappa_{i_\beta}^+$ chain condition. So letting $D \in M_{\On}$ be any dense subset of $\bar{\PP}$, we can find some predense subset $D'\subset D$ of cardinality $\leq \kappa_{i_\beta}$. Since $\bar{\PP}\subset H_{\kappa_{i_\beta}^+}^{M_{\On}}$, we know $D'\in H_{\kappa_{i_\beta}^+}^{M_{\On}}$ and hence that $D'$ can be coded using a canonical Skolem function as a set $S\subset\kappa_{i_\beta}$ in $M_{\On}$. Since $S\subset \kappa_{i_\beta}$, it follows immediately that $S$ is a class in the $\MK^{**}$ model $\mathcal{H}_\beta$. Hence, $D'$ is a definable hyperclass over $\mathcal{H}_\beta$. So any filter $G$ which is generic in the sense of $\mathcal{H}_\beta$ will meet $D'$ and hence meet $D$. Since $D$ was arbitrary, any such filter $G$ is generic in the sense of $M_\infty$.
			
		\end{proof}
		
		In particular, this claim means that $T^{\kappa_{i_\beta}}$ is generic over $\mathcal{H}_\beta$ in the sense of hyperclass forcing, because Corollary \ref{Corollary set size machete genericity} tells us it is generic over $M_{\On}$ as a set.
		
		We now need a way to transfer this up to genericity of $T$ over $\PP$ in $M_{\On}$. To help us here, we need the following technical result.
		
		\begin{claim}
			Let $j \leq \On$. Let $i<j$ be such that we iterated the top measure of $M_i$ at stage $i$ in $\mathcal{I}$, and let $p\in \PP_{\kappa_i}$ be in the generic filter generated by $T^{\kappa_i}$ (over $M_i$, or equivalently over $M_{\On}$). If $j<\On$, then $\pi_{i,j}(p)\upharpoonleft \kappa_j$ is in the ultrafilter generated by $T^{\kappa_j}$. If $j=\On$ then $\pi_{i,j}(p)\upharpoonleft \On$ is in the ultrafilter generated by $T$.
		\end{claim}
		
		\begin{proof}
			We prove this by induction on $j$, but simultaneously for all $i$ below $j$. The notation in this proof gets a bit fiddly, so to simplify things, we shall introduce some shorthand. We unify the two cases $j<\On$ and $j=\On$ by interpreting $\kappa_{\On}$ as $\On$.
			
			Let $\pi=\pi_{i,j}$. Let $\tilde{C}=\pi_{0,i}(C)$, and $\tilde{T}=T^{\kappa_{i}}$. Let $\bar{C}=\pi_{0,j}(C)\cap \kappa_j$ and $\bar{T}=T^{\kappa_j}$. Notice that $\tilde{C}=\pi_{0,\On}(C)\cap \kappa_i=W\cap \kappa_i$, and likewise for $\bar{C}$. We abuse notation by identifying $\bar{T}$ and $\tilde{T}$ with their corresponding generic filters. Let $\bar{p}=\pi(p){\upharpoonleft} \kappa_j$.
			
			As usual, we can think of $p$ as a combination of two parts: a name $\dot{s}=(\dot{s}_\lambda)_{\lambda \in \tilde{C}}$ for some stem of the generic sequence we are adding, and another name $\dot{X}=(\dot{X}_\lambda)_{\lambda \in \tilde{C}}$ for a sequence of measure $1$ sets. We can extend this notation by writing $\dot{s}_\lambda=\pi(\dot{s})(\lambda)$ and $\dot{X}_\lambda$ for all $\lambda \in \bar{C}$; since $\pi(p)$ is an end extension, the two definitions of $\dot{s}_\lambda$ and $\dot{X}_\lambda$ agree for $\lambda<\kappa_i$. For clarity, we avoid writing $\dot{s}$ to denote the overall sequence, instead writing $\bar{s}$ to denote the sequence $\dot{s}=(\dot{s}_\lambda)_{\lambda\in \bar{C}}$ and $\tilde{s}$ for the sequence $(\dot{s}_\lambda)_{\lambda \in \tilde{C}}$. We define $\bar{X}$ and $\tilde{X}$ likewise. Note that $\pi(p)$ consists of $\bar{s}$ together with $\bar{X}$. We know that $\tilde{s}^{\tilde{T}}$ and $\tilde{X}^{\tilde{T}}$ both agree with $\tilde{T}$; we want to show that $\bar{s}^T$ and $\bar{X}^T$ agree with $\bar{T}$.
			
			Now $\tilde{s}$ has finite support which is bounded below $\kappa_i$, so $\pi(\tilde{s})=\bar{s}$ differs from $\tilde{s}$ only by an empty end extension. In particular, since $\tilde{s}$ agrees with $\tilde{T}$ and $\tilde{T}$ is contained in $\bar{T}$, the whole of $\tilde{s}$ is (forced to be) contained in $\bar{T}$.
			
			It is more challenging to show that $\bar{X}$ agrees with $\bar{T}$. Let $\lambda \in \bar{C}=W\cap \kappa_j$; say $\lambda \in W_\beta$. We need to show that the measure $1$ set named by $\dot{X}_\lambda$ is forced by $T_\lambda$ to contain all of the $\omega$ sequence $S_\beta^\lambda$ (except for any elements of $\dot{s}_\lambda$). This is immediate for $\lambda<\kappa_i$, since then $\dot{X}_\lambda \in \tilde{X}$ is unmoved by $\pi$. We also know that $\kappa_i\not \in \bar{C}$ (since it's a critical point of $\mathcal{I}$). So without loss of generality, assume $\lambda>\kappa_i$.
			
			By definition, $\dot{X}_\lambda$ is forced by $\pi(p)\upharpoonleft \lambda$ to be a measure $1$ subset of $\lambda$. Recall from Proposition \ref{proposition magidor forcing measure generating sets} that this means there is a measure $1$ set $Y \in M_j$ such that $\pi(p)\upharpoonleft \forces \dot{X}_\lambda \supset Y \cap \Sigma_\lambda$, where for any generic $G$,
			$$\Sigma_\lambda^G := \{\nu<\lambda: \forall \kappa \in \bar{C}\cap \lambda \setminus \nu, \, (\nu+1)\cap G_\kappa=\emptyset\}$$
			
			Decoding this for $G=T_\lambda$, we find that $\nu \in \Sigma_\lambda^{T_\lambda}$ if and only if $\nu<\lambda$ and there is no $\kappa \in \bar{C}$ (necessarily in $W_\gamma$ for some $\gamma<\alpha$) such that $\nu \in (\min S_\gamma^\kappa,\kappa)$. So in particular, every term of $S_\lambda^\beta$ is in $\Sigma_\lambda^{T{\upharpoonleft} \lambda}$ by definition of $S_\lambda^\beta$. So we will have proved the claim if we can show that $S_\lambda^\beta \subset Y$ for a suitable choice of $Y$.
			
			Let $\kappa_l$ be the first element of $S_\lambda^\beta$ (and notice that $\kappa_l<\lambda<\kappa_j$). Now $\dot{X}^*:=\pi_{i,l}(\tilde{X})(\lambda)$ is a name for a measure $1$ subset of $\kappa_l$, so $\dot{X}^*$ is forced to contain $Y^*\cap \Sigma_{\kappa_l}$ for some measure $1$ set $Y^*\in M_l$.
			
			By applying the inductive hypothesis with $l$ in place of $j$, $\pi_{i,l}(p)$ is compatible with $T_{\kappa_l}$, and so we can choose $q\leq \pi_{i,l}(p)\upharpoonleft \kappa_l$ compatible with $T_{\kappa_l}$ which chooses a specific set $Y^*\in M_l$ such that $\dot{X}^* \supset Y^* \cap \Sigma_{\kappa_l}$. Let $Y=\pi_{l,j}(Y^*)$. Then $Y$ is a measure $1$ subset of $\pi_{l,j}(\kappa_l)=\lambda$, and $\pi_{l,j}(q)=\pi_{l,\lambda}(q)\in \PP_\lambda$. By elementarity $\pi_{l,j}(q)$ forces that $\dot{X}_\lambda \supset Y \cap \Sigma_{\lambda}$. By the same argument as in the proof of Claim \ref{claim critical points in measure 1 set}, $S_\lambda^\beta \subset Y$.
			
			Applying the inductive hypothesis again, this time with $\lambda$ in place of $j$, we see that $\pi_{l,j}(q)$ is compatible with $T_\lambda$. Hence,
			
			$$\dot{X}_\lambda^{T_\lambda} \supset Y \cap \Sigma_\lambda^{T_\lambda}\supset S_\beta^\lambda$$
			
			With that, the claim is proved.
		\end{proof}
		
		Now, let $D\subset \PP$ be a dense definable hyperclass over $M_{\On}$. Let $\varphi(\vec{x})$ be the formula defining it. By an elementarity argument, we know that for large enough $\beta$, $\varphi(\vec{x})$ defines a hyperclass $D_\beta$ over $\mathcal{H}_\beta$, which is dense in the $M_{i_\beta}$ analogue of $\PP$, i.e. $\PP_{\kappa_{i_\beta}}$. Since $T^{\kappa_{i_\beta}}$ is generic for $\PP_{\kappa_{i_\beta}}$ over $\mathcal{H}_\beta$, it must meet $D_{i_\beta}$. Let $p\in T^{\kappa_{i_\beta}} \cap D_{i_\beta}$. Then ${\mathcal{H}_\beta \vDash \varphi(\vec{x})(p)}$. By elementarity, $\varphi(\vec{x})(\pi_{i,\On}(p))$ holds in the $\MK^{**}$ model with set-domain $M_{\On}$, so ${\pi_{i,\On}(p)\in D}$. But $T^{\kappa_{i_\beta}}=T{\upharpoonleft} \kappa_{i_\beta}$, so by the claim we just proved, $\pi_{i,\On}(p)\in T{\upharpoonleft} \On=T$. So $T$ meets every dense subhyperclass of $\PP$ over the $\MK$ model extending $M_{\On}$, and is therefore generic over that model.
	\end{proof}
	
	\subsection{The generic extension gives $L[R]$}
	\noindent
	\\
	\noindent At this point, we have shown that $M_{\On}[T]$ is (the set part of) a well-defined hyperclass generic extension over the $\MK^**$ model $(M_{\On},\mathcal{C})$. It remains to show $M_{\On}[T]=L[R]$.
		
	First, let us make some easy observations. Abusing notation slightly, we shall write $\cup T$ to denote the collection of all terms in sequences in $T$, i.e. all $\mu$ such that $(n,\mu)\in S^\lambda_\beta$ for some $n,\lambda$ and $\beta$. By definition of $T$, $\cup T \subset R$.
	
	\begin{lemma}
		$R\setminus {\cup T}$ is finite.
	\end{lemma}
	\begin{proof}
		Clearly, for any $\beta<\alpha$, $R_\beta\setminus {\cup T}$ will just consist of the top few elements of $R_\beta$, above which there is no element of $W_\beta$. Since $W_\beta$ is defined as the class of $\omega$ limits of $R_\beta$, there can only be finitely many of these for any given $\beta$.
		
		Let $X$ be the set of all $\beta<\alpha$ such that $R_\beta \setminus {\cup T}\neq \emptyset$. As we've seen, if this holds for some $\beta$ then $R_\beta$ has a finite (but nonempty) final segment, and so contains a maximum element $\lambda_\beta$.
		
		If $\gamma,\beta \in X$ and $\gamma<\beta$, then $\lambda_\beta$ is a limit of elements of $R_\gamma$, and so $\lambda_\gamma>\lambda_\beta$. Hence, $X$ must be finite, or $(\lambda_\beta)_{\beta\in X}$ would be an infinite decreasing sequence of ordinals.
		
		So $R\setminus {\cup T} = \bigcup_{\beta \in S} (R_{\beta}\setminus {\cup T})$ is a finite union of finite sets, and hence finite.
	\end{proof}
	
	\begin{lemma}\label{lemma R and T definability}
		If $M$ is any model of $\ZFC$, the $T$-definable sets over $M$ are precisely the same as the $R$-definable sets over $M$.  In particular, if $M$ is closed under $T$-definability then it contains $L[R]$ as a subclass, definable in any language which includes $T$. Likewise, if $M$ is closed under $R$-definability and contains $M_{\On}$ as an $R$-definable subclass, then it also contains $M_{\On}[T]$ as an $R$-definable subclass.
	\end{lemma}
	\begin{proof}
		$R$ is simply $\cup T$ together with a finite collection of extra ordinals, so any formula in terms of $R$ can easily be turned into one in terms of $T$. Conversely, in a model of $\ZFC$ there is a class function taking any set $S$ to the Cantor-Bendixson rank of its largest element. We can apply this class function to the $R$ definable set $R\cap (\kappa+1)$ to determine the Cantor-Bendixson rank of any $\kappa \in R$, and thus calculate the classes $R_\beta$ for $\beta<\alpha$ within $M$. Once we have done that, it is easy to find a formula to calculate where, if anywhere, a given $\kappa$ will appear in $T$, and thus turn any formula in terms of $T$ into one in terms of $R$.
	\end{proof}
		
	\begin{lemma}\label{lemma M_infty contained in L[Reg]}
		(The set part of) $M_{\On}$ is an $R$-definable subclass of $L[R]$.
	\end{lemma}
	\begin{proof}
		Recall that by definition, the set part of $M_{\On}$ is $(J_{\On}^E, E)=(L[E],E)$, where $E$ is the extender sequence of $M_{\On}$. So it suffices to show that $L[R]$ can calculate $E$, i.e. that it can determine which ordinals $E$ believes are measurable, and which of the sets in $M_{\On}$ are measure $1$. We can then carry out the usual recursive construction to calculate $L_\beta[E]$ for all $\beta$, and hence find $M_{\On}$. Note that we do \textit{not} need to devise an explicit test for whether a given set is an element of $M_{\On}$; it suffices to give a way to determine whether a set we already know is in $M_{\On}$ is measurable.
		
		The class of measurables of $E$ is just $\pi_{0,\On}(C)= \bigcup_\beta W_\beta=W$. This is clearly definable over $L[R]$.
		
		Now we must find a way to express in $L[R]$ the statement, for $\lambda \in W$ and $X\in L[R]$,
		
		\begin{quotation}
			``If $X\in L[E]$ then $X$ is a measure $1$ subset of $\lambda$."
		\end{quotation}
		
		Of course, it is easy to express ``$X$ is a subset of $\lambda$". The challenge is in the ``measure $1$" part. Suppose that $X\in L[E]$, and that (say) $\lambda \in W_\beta$. If $X$ is measure $1$, then by Theorem \ref{theorem Mathias Magidor criterion}\footnote{Invoking genericity of $T$ and Ben Neria's result is actually very much overkill for showing this simple fact. It can be proved directly with an adjusted version of Claim \ref{claim critical points in measure 1 set}. But since we've proved genericity already, we might as well save ourselves some time by making use of it.} and the fact that $T$ is generic over $L[E]$, $X$ will contain all but finitely many terms of $S^\beta_\lambda$.
		
		On the other hand, if $X\subset \lambda$ is not measure $1$ but is an element of $L[E]$, then $\lambda \setminus X$ will be measure $1$ instead. So as we just saw, $\lambda \setminus X$ will contain all but finitely many terms of $S^\epsilon_\lambda$, and therefore $X$ itself will contain at most finitely many terms of $S^\epsilon_\lambda$.
		
		So if $X\in L[E]$ is a subset of $\lambda$, then $X$ is measure $1$ if and only if it contains all but finitely many terms of $S^\beta_\lambda$, which is a statement which can be expressed in $L[R]$. So $L[R]$ knows what $E$ looks like, and can therefore define (the set part of) $M_{\On}$ in terms of $R$.
	\end{proof}
	
	\begin{corollary}
		$M_{\On}[T]$ is precisely $L[R]$, and its extender sequence is $R$-definable over $L[R]$.
	\end{corollary}
	\begin{proof}
		Applying Lemma \ref{lemma R and T definability} with $M=M_{\On}[T]$ (which is closed under $T$-definability) gives that $L[R]$ is a $T$-definable subclass of $M_{\On}[T]$. Conversely, if we apply Lemma \ref{lemma R and T definability} with $M=L[R]$, which is closed under $R$-definability and which Lemma \ref{lemma M_infty contained in L[Reg]} shows contains $M_{\On}$ as an $R$-definable subclass, we get that $L[R]$ contains $M_{\On}[T]$ as an $R$-definable subclass. Combining these, we see that
		$$L[R]\subseteq M_{\On}[T]\subseteq L[R]$$
	\end{proof}	
	This ends the proof of Theorem \ref{theorem L[R]}.

\section{Finding Machetes}

We end this paper with an existence result about O-Machetes. So far, we've been assuming we've already found an O-Machete, and then used it to generate some structure $L[R]$. But maybe we could do the opposite: we could start with some interesting $L[S]$ type structure for some predicate $S$, and show that every nice enough O-Machete must be an element of that model.

The predicate $S$ we are going to look at is $\Reg$, a natural choice when we want to find a model which contains mice that we've just seen can generate $L[\Reg_{<\gamma}]$. Of course, it's possible that $\Reg$ contains nothing interesting and doesn't give us any useful information: for example, if there are no inaccessibles then $\Reg$ is effectively just $\Card$. So we need to assume some largeness criterion for $\Reg$. The condition we'll be using is ``$\Reg$ is a stationary class'', which can also be expressed as ``$\On$ is Mahlo''.

\begin{theorem}\label{theorem friendly machetes}Suppose that $\On$ is Mahlo; i.e., that $\Reg$ is stationary. Then for any mouse $M$, $M <^* \Mmuind \rightarrow M<^*  
K^{L[\Reg]}$.\end{theorem}

\begin{proof}
	Let $M_0$ be $K^{L[\Reg]}$. Let $N_0$ be in the $<^*$ least equivalence class of (set) mice which is $>^* M_0$. (If no such class exists, then the theorem is trivial.) Since $M_0$ is a weasel, it suffices to prove that $\Mmuind \leq^* N_0$.  Note that $N_0$ is active, because otherwise we could cut it down to its largest measurable and the result would still be $>^* M_0$.
	
	Assume (seeking a contradiction) that $N_0 <^* \Mmuind$. Then we can find some formula $\varphi(v_1,\ldots,v_n,v_{n+1})$, and some parameters $\gamma_1<\ldots<\gamma_n<\nu$ such that
	$$H^{N_0}_{\nu^+}\vDash \varphi(\gamma_1,\ldots,\gamma_n,\nu)$$
	but such that there is no measurable $\lambda<\nu$ of $N_0$ for which
	$$H^{N_0}_{\lambda^+}\vDash \varphi(\gamma_1,\ldots,\gamma_n,\lambda)$$
	
	Let $(\mathcal{I},\mathcal{J})=(\langle M_i,\pi_{i,j}\rangle_{i\leq j \leq \On},\langle N_i,\tau_{i,j}\rangle_{i\leq j \leq \On})$ be the coiteration of $M_0$ and $N_0$. As usual, we write $\kappa_j$ for the $j$'th critical point of this coiteration. Note that neither $\mathcal{I}$ nor $\mathcal{J}$ involves any cut-downs: $\mathcal{I}$ because it is a simple iteration, and $\mathcal{J}$ because any cut-down would lead to a smaller mouse which was still $>^* M_0$. Similarly, the top measure $\nu$ of $N_0$ will be sent precisely to $\On$ by $\mathcal{J}$, and everything below $\nu$ will be ``left behind'' at some stage of the iteration.

	Fix some $i\in \On$ which is large enough that $\mathcal{J}$ has already done everything interesting by stage $i$:
	
	\begin{enumerate}
		\item $i\in \On\setminus N_0$;
		\item For $1\leq k \leq n$, $\tau_{0,i}(\gamma_k)=\tau_{0,\infty}(\gamma_k)=:\widebar{\gamma}_k$.
	\end{enumerate}
	
	To avoid some annoying technicalities, we also choose $i$ to be inaccessible in $V$, and such that $\kappa_i$ is the top measurable of $\bar{N}_i$. (This is easy to do because $\mathcal{J}$ iterates the top measure of $N$ club many times and $\Reg^V$ is stationary.)
	
	We must now perform a small, but rather notationally heavy, technical step. The weasel $M_i$ is not necessarily definable over $L[\Reg]$. For the final part of the proof to work, we need to adjust the $\mathcal{I}$ side of the coiteration so we find our way back into $L[\Reg]$ at some point after stage $i$. We also need to adjust $\mathcal{J}$ so that the two iterations agree.
	
	The dual iteration $$(\tilde{\mathcal{I}},\tilde{\mathcal{J}})=(\langle \tilde{M_j},\tilde{\pi}_{j,k}\rangle_{j,k<\leq \On},\langle \tilde{N_j},\tilde{\tau}_{j,k}\rangle_{j,k<\leq \On})$$
	we construct will consist of three parts:
	\begin{enumerate}
		\item The initial segment $({\mathcal{I}\upharpoonleft i}, {\mathcal{J}\upharpoonleft i})=(\langle M_j,\pi_{j,k}\rangle_{j,k\leq i},\langle N_j,\tau_{j,k}\rangle_{j,k\leq i})$ of $(\mathcal{I},\mathcal{J})$.
		\item An iteration $\langle \tilde{M}_j,\tilde{\pi}_{j,k}\rangle_{i\leq j,k \leq \tilde{i}}$ (for some $\tilde{i}\in \Reg$) of $M_i=\tilde{M}_i$ to some weasel $\tilde{M}_{\tilde{i}}$ which is definable over $L[\Reg]$, together with a corresponding iteration of $N_i=\tilde{N}_i$, such that $\tilde{N}_{\tilde{i}}$ agrees with (i.e. has the same domain and measures as) $\tilde{M}_{\tilde{i}}$ below the top measure $\kappa_{\tilde{N}_{\tilde{i}}}$ of $\tilde{N}_{\tilde{i}}$.
		\item The coiteration $\langle \tilde{M}_j,\tilde{\pi}_{j,k}\rangle_{\tilde{i}\leq j,k \leq \On}$ of $\tilde{M}_{\tilde{i}}$ and $\tilde{N}_{\tilde{i}}$. This will only iterate measurables greater than or equal to the top measure of $\tilde{N}_{\tilde{i}}$, since $\tilde{N}_{\tilde{i}}$ agrees with $\tilde{M}_{\tilde{i}}$ below its top measure $\kappa_{\tilde{N}_{\tilde{i}}}$. That is, $\tilde{M}_{\tilde{i}}\restriction \kappa_{\tilde{N}_{\tilde{i}}}
 = \tilde{N}_{\tilde{i}}\restriction \kappa_{\tilde{N}_{\tilde{i}}}$.
	\end{enumerate}
	
	Note that $\mathcal{I}$ and $\mathcal{J}$ are not necessarily normal iterations, although each of their three pieces will be normal.
	
	The first pieces of $\tilde{\mathcal{I}}$ and $\tilde{\mathcal{J}}$ are already defined; for $j\leq i$ take $\tilde{M}_j=M_j$ and $\tilde{N}_j=N_j$. We define the second piece of $\tilde{\mathcal{I}}$ as follows. Take some sufficiently long universal iteration of $M_0$ in which is definable over $L[\Reg]$. Then we can simply iterate $\tilde{M}_i$ as described in \cite{welch_HaertigPaper}[2.9] to produce some term $\tilde{M}_{\tilde{i}}$ of this iteration, which must also be definable over $L[\Reg]$. Without loss of generality, we can also take $\tilde{i}$ to be regular in $V$.
	
	Next we define the second piece of $\tilde{J}$, using the second part of $\tilde{I}$ as a guide. Note that $N_i=\tilde{N}_i$ and $M_i=\tilde{M}_i$ agree below $\kappa_i$, which by assumption is the top measure $\kappa_{\tilde{N}_i}$ of $\tilde{N}_i$. We define $\tilde{\mathcal{J}}\upharpoonleft [i,\tilde{i}]$ such that for all $i\leq j \leq \tilde{j}$, $\tilde{N}_j$ agrees with $\tilde{M}_j$ below the former's top measure $\tilde{\tau}_{i,j}(\kappa_i)$. To do this, if $i\leq j <\tilde{i}$ and $\tilde{M}_{j+1}=\Ult(\tilde{M}_j,U)$ we simply define:
	\begin{equation*}
		\tilde{N}_{j+1}:=\begin{cases}
			\Ult(\tilde{N}_j,U) &\text{if} \crit(U)<\tilde{\tau}_{i,j}(\kappa_i)\\
			\tilde{N}_j &\text{if} \crit(U)>\tilde{\tau}_{i,j}(\kappa_i)
		\end{cases}
	\end{equation*}
	
	Limit stages are handled by taking direct limits. It is easy to show by induction that for $i\leq j \leq \tilde{i}$, $\tilde{M}_j$ and $\tilde{N}_j$ agree below the latter's largest measurable.
	
	Finally, we define the third parts of $\tilde{\mathcal{I}}$ and $\tilde{\mathcal{J}}$ by coiterating $\tilde{M}_{\tilde{i}}$ and $\tilde{N}_{\tilde{i}}$ in the usual way.
	
	Let $\tilde{\kappa}_j$ denote the $j$th critical point of $\tilde{\mathcal{I}}$ and $\tilde{\mathcal{J}}$, noting that the two iterations have the same critical points. The $\tilde{\kappa}_j$ sequence is strictly increasing except at $i$ and $\tilde{i}$. For $1\leq k \leq n$, let $\tilde{\gamma}_k=
	\tilde{\tau}_{i,\tilde{i}}(\widebar{\gamma}_k)=\tilde{\tau}_{0,\tilde{i}}(\gamma_k)$. Since $\widebar{\gamma}_k<\kappa_{N_i}$, we know $\tilde{\gamma}_k<\kappa_{\tilde{N}_{\tilde{i}}}=\tilde{\kappa}_{\tilde{i}}$, and hence $\tilde{\gamma}_k=\tilde{\tau}_{0,\On}(\gamma_k)$ and also $\tilde{\pi}_{\tilde{i},\On}(\tilde{\gamma}_k)=\tilde{\gamma}_k$. By elementarity, $\tilde{M}_{\On}=\tilde{N}_{\On}\restriction \On$ contains no measurables $\lambda$ such that
	$$H^{\tilde{M}_{\On}}_{\lambda^+}\vDash \varphi(\tilde{\gamma}_1,\ldots,\tilde{\gamma}_n,\lambda)$$
	and the same is thus also true of $\tilde{M}_{\tilde{i}}$.

	We shall now show that $L[\Reg]$ can \textit{almost} calculate the critical points of the iteration $\tilde{\mathcal{J}}$ at which we iterate the top measure. But only ``almost'': the test we use misses some of the critical points.
	
	\begin{definition}
		Let us call an ordinal $\delta$ \textit{useful} if:
		
		\begin{itemize}
			\item $\delta>\sup_{j<\tilde{i}}\tilde{\kappa}_j$ and $\delta>\lvert\tilde{N}_{\tilde{i}}\rvert$;
			\item $H^{\tilde{M}_{\tilde{i}}}_{\delta^+} \vDash\varphi(\tilde{\gamma}_1,\ldots,\tilde{\gamma}_n,\delta)$;
			\item $\delta \in \Reg$.
		\end{itemize}

	\end{definition}
	Notice that $L[\Reg]$ can calculate which ordinals are useful, since it knows about $\tilde{M}_{\tilde{i}}$ and $\Reg$.
	\begin{lemma}\label{Lemma useful ordinals are critical points}
		Let $\delta\in \On$ be useful. Then $\tilde{\pi}_{\tilde{i},\On}(\delta)=\delta$, and $\delta=\tilde{\kappa}_\delta=\kappa_{\tilde{N}_\delta}$ is a critical point of $\tilde{\mathcal{J}}$ where we iterated the top measure.
	\end{lemma}
	
	\begin{proof}
		Let $j\in \On$ be least such that $j\geq \tilde{i}$ and $\tilde{\kappa}_j\geq \delta$. Since $\delta$ is strongly inaccessible in $V$, we know $\tilde{\pi}_{\tilde{i},j}(\delta)=\delta$ by Lemma \ref{lemma V cardinals fixed points}. Also, since $H^{\tilde{M}_{\tilde{i}}}_{\delta^+} \vDash\varphi(\tilde{\gamma}_1,\ldots,\tilde{\gamma}_n,\delta)$, we know that $\delta$ can't be measurable in $\tilde{M}_{\tilde{i}}$, and hence can't be measurable in $\tilde{M}_j$ either. So either $\tilde{\kappa}_j>\delta$, or $\tilde{\kappa}_j=\delta$ is a ``do nothing'' step of $\tilde{\mathcal{I}}$. In either case, $\tilde{\pi}_{j,\On}(\delta)=\delta$ and hence of $\tilde{\pi}_{\tilde{i},\On}(\delta)=\delta$.		
		
		Now, $\delta$ is strongly inaccessible in $V$ and larger than $\lvert\tilde{N}_{\tilde{i}}\rvert$, so $k<\delta$ implies $\lvert \tilde{N}_k\rvert<\delta$. Hence by \ref{Lemma V cardinals end up with a measurable on them} we know that $\delta=\sup_{k<\delta}\tilde{\kappa}_k$ is measurable in $\tilde{N}_{\delta}$, and $\tilde{\kappa}_\delta\geq \delta$. But we also know that $\delta=\tilde{\pi}_{\tilde{i},\On}(\delta)$ isn't measurable in $\tilde{N}_{\On}$ since it wasn't measurable in $\tilde{M}_{\tilde{i}}$, and hence $\tilde{\kappa}_\delta=\delta$ is a (nontrivial) critical point of $\tilde{\mathcal{J}}$. Finally, 
		$$H_{\delta^+}^{\tilde{N}_\delta}=H_{\delta^+}^{\tilde{N}_{\On}}=H_{\delta^+}^{\tilde{M}_{\tilde{i}}}\vDash \varphi(\tilde{\gamma}_1,\ldots,\tilde{\gamma}_n,\delta)$$
		But $\tilde{N}_\delta=^*N_0$ believes there are no measurables $\delta'$ with this property other than its top measure, and hence $\delta$ must be the top measure of $\tilde{N}_\delta$.
		
	\end{proof}

	Conversely, any sufficiently nice critical point is useful, and thus we can find as many useful ordinals as we need.

	\begin{lemma}
		There are unboundedly many useful ordinals.
	\end{lemma}
	
	\begin{proof}
		In the iteration $\tilde{N}_{\tilde{i}}\rightarrow \tilde{N}_{\On}$ we iterate the top measure a club class of times. For any such top-measure iteration point $\tilde{\kappa}_h$,
		
		\begin{align*}
		H^{\tilde{N}_{h}}_{\tilde{\kappa}_h^+} &\vDash\varphi(\tilde{\gamma}_1,\ldots,\tilde{\gamma}_n,\tilde{\kappa}_h)\\
		H^{\tilde{N}_{\On}}_{\tilde{\kappa}_h^+} &\vDash\varphi(\tilde{\gamma}_1,\ldots,\tilde{\gamma}_n,\tilde{\kappa}_h)\\
		H^{\tilde{M}_{h}}_{\tilde{\kappa}_h^+} &\vDash\varphi(\tilde{\gamma}_1,\ldots,\tilde{\gamma}_n,\tilde{\kappa}_h)
		\end{align*}
		
		Hence, $\tilde{\kappa}_h$ cannot be measurable in $\tilde{M}_h$, since $\tilde{M}_h$ contains no measurables which believe $\varphi$. In particular, $\tilde{\kappa}_h$ is not, for any $\tilde{i}<j<h$, the image under $\tilde{\pi}_{j,h}$ of some measurable in $\tilde{M}_j$.
		
		Since there are a club of such $\tilde{\kappa}_h$ and $\Reg^V$ is stationary, we can find unboundedly many such $\tilde{\kappa}_h\in \Reg$. For such ordinals, Lemma \ref{lemma V cardinals fixed points}, tells us that $\tilde{\kappa}_h$ is a fixed point of $\tilde{\pi}_{\tilde{i},h}$, and so.
		$$H^{\tilde{M}_{\tilde{i}}}_{\tilde{\kappa}_h^+} \vDash\varphi(\tilde{\gamma}_1,\ldots,\tilde{\gamma}_n,\tilde{\kappa}_h)$$
		
		Any sufficiently large such ordinal therefore satisfies the criteria for usefulness.	
	\end{proof}

	As remarked above, ``$\delta$ is a useful ordinal" can be expressed as a statement about $\tilde{M}_{\tilde{i}}$, $\delta$ and $\Reg$, so $L[\Reg]$ can identify the class of all useful ordinals. For $\alpha\in \On$ let $\delta_\alpha$ be the $\alpha$'th useful ordinal. Let $C$ be the club of all limits of useful ordinals, and (again exploiting the fact that $\Reg$ is stationary) let $\delta \in C\cap \Reg$. Then (since $\delta$ is inaccessible) we know $\delta=\delta_\delta=\sup_{\alpha<\delta} \delta_\alpha$.
	
	Lemma \ref{Lemma useful ordinals are critical points} we know that $\delta$ is the top measure of $\tilde{N}_\delta$, and hence that
	$$H^{\tilde{N}_\delta}_{\delta^+}\vDash \varphi(\tilde{\gamma}_1,\ldots,\tilde{\gamma}_k,\delta)$$
	But the same lemma also tells us that $\delta=\tilde{\kappa}_\delta$, and hence that $H^{\tilde{N}_\delta}_{\delta^+}=H^{\tilde{N}_{\On}}_{\delta^+}=H^{\tilde{M}_{\On}}_{\delta^+}$. Finally, Lemma \ref{Lemma useful ordinals are critical points} also tells us that $\delta=\tilde{\pi}_{\tilde{i},\On}(\delta)$. Putting this together, we find that
	$$H^{\tilde{M}_{\tilde{i}}}_{\delta^+}\vDash\varphi(\tilde{\gamma}_1,\ldots,\tilde{\gamma}_k,\delta)$$
	
	To find our promised contradiction, we will show that $\delta$ is measurable in $\tilde{M}_{\tilde{i}}$.	
	
	Since $\delta$ has the same subsets in $\tilde{N}_\delta$ and $\tilde{M}_{\On}$ we know that the normal measure $U$ on $\mathcal{P}(\delta)\cap \tilde{N}_\delta$ in $\tilde{N}_\delta$ is also a normal measure on $\mathcal{P}(\delta)\cap\tilde{M}_{\On}$. (Of course, $\delta$ isn't actually measurable in $M_{\On}$, but that's because the $M_{\On}$ doesn't know about $U$, not because anything has changed about the subsets of $\delta$.)
	
	So the pull-back $U^*$ of $U$ to $\tilde{M}_{\tilde{i}}$, defined by
	$${X\in U^*} \iff {X\in \tilde{M}_{\tilde{i}}}\wedge {X\subset \delta} \wedge {\tilde{\pi}_{\tilde{i},\On}(X)\in U}$$
	is a normal measure on $\delta=\tilde{\pi}_{\tilde{i},\On}^{-1}(\delta)$ over $\tilde{M}_{\tilde{i}}$. So we've found a normal measure on $\delta$ in $\tilde{M}_{\tilde{i}}$. It remains to show that $U^*$ is in the extender sequence of $\tilde{M}_{\tilde{i}}$.
	
	\begin{lemma}A subset $X \in \tilde{N}_\delta$ of $\delta$ is measure $1$ if and only if it contains a tail of the sequence $(\delta_\alpha)_{\alpha<\delta}$. Hence, $L[\Reg]$ can identify whether a subset of $\delta$ in $\tilde{N}_h\cap \tilde{M}_{\tilde{i}}$ is measure $1$ or not.\footnote{Note the similarities between this lemma and the proof of Lemma \ref{lemma M_infty contained in L[Reg]}. The only change is that we're now dealing with a $\delta$ sequence of critical points, whereas back there it was just an $\omega$ sequence. This makes no difference to the underlying proof, but it means we can't use the same shortcut we did in that lemma.}
	\end{lemma}
	\begin{proof}
		Suppose that $X\in \tilde{N}_\delta$ is measure $1$. Let $f\in \tilde{N}_{\tilde{i}}$ and $j_0<\ldots<j_m<\delta$ be such that $X=\tilde{\tau}_{\tilde{i},\delta}(f)(\tilde{\kappa}_{j_0},\ldots,\tilde{\kappa}_{j_m})$. Let $\alpha<\delta$ be large enough that $\delta_\alpha>\tilde{\kappa}_{j_m}$. Now we know by Lemma \ref{Lemma useful ordinals are critical points} that $\delta_\alpha=\tilde{\kappa}_{\delta_\alpha}$ and $\tilde{\tau}_{\delta_\alpha,\delta}(\delta_\alpha)=\delta$. So by elementarity, $X_\alpha:=\tilde{\tau}_{\tilde{i},\delta_\alpha+1}(f)(\tilde{\kappa}_{j_0},\ldots,\tilde{\kappa}_{j_m})$ is a measure $1$ subset of $\tilde{\tau}_{\delta_\alpha,\delta_\alpha+1}(\delta_\alpha)$. So $\delta_\alpha\in X_\alpha$. But then 
		$$\delta_\alpha=\tilde{\tau}_{\delta_\alpha+1,\delta}(\delta_\alpha)\in \tilde{\tau}_{\delta_\alpha+1,\delta}(X_\alpha)=X$$
		
		\noindent So $X$ contains $\delta_\alpha$ for all sufficiently large $\alpha<\delta$.
		
		On the other hand, if $X\in \tilde{N}_\delta$ is not measure $1$, then $\delta \setminus X$ is measure $1$ instead. So $\delta\setminus X$ contains a tail of the $\delta_\alpha$, and therefore $X$ does not contain such a tail.
	\end{proof}
	
	By the preceding lemma, and the fact that $\mathcal{P}(\delta)$ is the same in $\tilde{N}_\delta$ and $\tilde{M}_{\On}$ we know for $X\in \tilde{M}_{\tilde{i}}$ that
	\begin{align*}
		X\in U^* &\iff X\subset \delta \wedge \tilde{\pi}_{\tilde{i},\On}(X)\in U\\
		&\iff X\subset \delta \wedge \tilde{\pi}_{\tilde{i},\On}(X)\text{ contains a tail of }(\delta_\alpha)_{\alpha<\delta}\\
		&\iff X\subset \delta \wedge \exists \beta<\delta \  \forall \alpha \in (\beta,\delta)\, \delta_\alpha \in \tilde{\pi}_{\tilde{i},\On}(X)\\
		&\iff X\subset \delta \wedge \exists \beta<\delta \ \forall \alpha \in (\beta,\delta) \tilde{\pi}_{\tilde{i},\On}(\delta_\alpha)\, \in \tilde{\pi}_{\tilde{i},\On}(X)\\
		&\iff X\subset \delta \wedge \exists \beta<\delta \ \forall \alpha \in (\beta,\delta)\, \delta_\alpha \in X
	\end{align*}
	
	The penultimate line follows because we know $\tilde{\pi}_{\tilde{i},\On}(\delta_\alpha)=\delta_\alpha$  (see Lemma \ref{Lemma useful ordinals are critical points}). Since the final line can be expressed within $L[\Reg]$, we have shown thta $U^*\in L[\Reg]$.
	
	Finally, note that within $L[\Reg]$ (or indeed $V$) $\delta>\omega$ is regular, so $U^*$ is $\omega$ complete. Further in $L[Reg]$, $\tilde{M}_{\tilde{i}}$ is an iterate of a universal weasel, by measures below $\delta$. Hence, by \cite[7.3.7]{zemanInnerModelsAndLargeCardinals} we know $U^*$ will appear on the extender sequence of $\tilde{M}_{\tilde{i}}$. So $\tilde{M}_{\tilde{i}}$ contains a measurable $\delta$ such that
	$$H_{\delta^+}^{\tilde{M}_{\tilde{i}}}\vDash \varphi(\tilde{\gamma}_1,\ldots,\tilde{\gamma_n},\delta)$$
		Contradiction!
\end{proof}

\begin{corollary}
	Assume $\Reg$ is Mahlo. Then $\Ome\in K^{L[\Reg]}$ for all $\alpha\leq \infty$.
\end{corollary}

\begin{proof}
This is a standard argument.		Proposition \ref{proposition machete below muind} tells us that $\Ome <^* \Mmuind$, so Theorem \ref{theorem friendly machetes} says $\Ome <^* K^{L[\Reg]}$. Let $\rho = \rho_1^{\Ome}$. Recall from Proposition \ref{proposition basic properties of machetes} that $\rho_1^{\Ominf} = \omega$, and that if $\alpha<\infty$ then $\rho_1^{\Ome}\leq \alpha$ and $\Ome$ contains no measurables $\leq \alpha$. Let $p$ be the first standard parameter of $\Ome$, and let $A^{1,p}\subset \omega \times \mbox{ }^{<\omega}[\rho]$ be its first standard code.

The coiteration of $\Ome$ with $K^{L[\Reg]}$ begins by iterating the  least measurable of $K^{L[\Reg]}$ below $\alpha$ (if any) until it is greater than or equal to the least measurable of $\Ome$, which we know is above $\rho$. The resulting weasel, which we may call $M_i$, is definable within $K^{L[\Reg]}$ (since we just gave an explicit process for generating it). But now $A^{1,p}\in \Sigma_\omega(\Ome)\cap \mathcal{P}(\rho) \subset \Sigma_\omega(M_i)\cap \mathcal{P}(\rho) \subset K^{L[\Reg]}$. Hence $A^{1,p}$, and so $\Ome $, are in
$ K^{L[\Reg]}$.
	
\end{proof}

\section{Open Questions}





We saw in Section 5 that all mice below $\Mmuind$ are $<^* K^{L[\Reg]}$ (assuming $\Reg$ is rich enough). However, the converse is not true. Indeed, one can show that  the methods of Theorem \ref{theorem friendly machetes} apply using  $\Mmuindtwo$,   the least sound mouse which reflects all second order formulae, to give that $\Mmuind <^* K^{L[\Reg]}$. Similar results hold for languages of higher orders.

\begin{question}
	Assume $On$ is Mahlo. Can we characterise the least mouse $M$ such that $K^{L[\Reg]} <^* M$?
\end{question}

Theorem \ref{theorem L[R]} shows that $L[\Reg\setminus \Hype]$ can be generated from a machete mouse, and again it is possible to adapt our techniques to show (for example) that $L[\Reg \setminus \Hype(\Hype)]$ can be generated in the same way, where $\Hype(\Hype)$ denotes the ``hyperhyperinaccessibles'', i.e. those hyperinaccessibles which are equal to their own CB rank in $\Hype$. In general, it seems that given a sufficiently large mouse, we can generate $L[\Reg \setminus H]$ for many suitable classes $H$ of hyperinaccessible cardinals in $\Reg$. However, it is open whether we can eliminate this definable class $H$ entirely.

\begin{question} Assume $On$ is Mahlo.
	Can we prove that $L[\Reg]$ is a generic extension of an iterate of O-Sword? 
\end{question}

It is known by current work of Schindler {\em  et al.}, that any mouse $M$ that projects to $\omega$ and is $<^\ast $ the least active mouse $N$ in which the $N$-measurables below $\kappa_N$ are $N$-stationary, is in $C^\ast$. Here  $C^\ast$  is the ``$\cof$-$\omega$'' model, $L[\cof_\omega]$ \cite{kennedyMagidor_innermodels}. Hence:

\begin{question} Assume $On$ is Mahlo.
What is the relationship between the reals of $C^\ast$ and $L[\Reg]$?	Is $\Mmuind $ in $C^\ast$?	
\end{question}

It is easy to see that the mouse $\Mmuind$ is $>^\ast$ than the mouse $ N$ just mentioned.

\bibliographystyle{plain}
\bibliography{references}
\end{document}